\documentclass[reqno,11pt]{amsart}
\usepackage{amsmath,amssymb,latexsym,cancel,rotating,hyperref}

\usepackage{eucal}

\usepackage{amscd}
\usepackage[dvips]{color}
\usepackage{multicol}
\usepackage[all]{xy}           
\usepackage{graphicx}
\usepackage{color}
\usepackage{colordvi}
\usepackage{xspace}
\usepackage{tikz}

\numberwithin{equation}{section}

\newtheorem{theorem}{Theorem}[section]
\newtheorem{proposition}[theorem]{Proposition}
\newtheorem{conjecture}[theorem]{Conjecture}

\newtheorem{lemma}[theorem]{Lemma}

\newtheorem{remark}[theorem]{Remark}

\newtheorem{definition}[theorem]{Definition}

\input xy
\xyoption{poly}
\xyoption{2cell}
\xyoption{all}

\renewcommand{\eqref}[1]{{\rm (\ref{#1})}}

\begin{document}
	
	\title[On a conjecture of transposed Poisson $n$-Lie algebras]
	{On a conjecture of transposed Poisson $n$-Lie algebras}
	
	\author{Junyuan Huang, Xueqing Chen, Zhiqi Chen and Ming Ding}
	\address{School of Mathematics and Information Science\\
		Guangzhou University, Guangzhou 510006, P.R.China}
	\email{jy4545@e.gzhu.edu.cn (J.Huang)}
	\address{Department of Mathematics,
		University of Wisconsin-Whitewater\\
		800 W. Main Street, Whitewater, WI.53190. USA}
	\email{chenx@uww.edu (X.Chen)}
	\address{School of Mathematics and Statistics\\
 Guangdong University of Technology, Guangzhou 510520, P.R. China} \email{chenzhiqi@nankai.edu.cn (Z.Chen)}
	\address{School of Mathematics and Information Science\\
		Guangzhou University, Guangzhou 510006, P.R.China}
	\email{dingming@gzhu.edu.cn (M.Ding)}



	
	

	\begin{abstract}
      In this paper, we obtain a rich family of identities for transposed Poisson $n$-Lie algebras, and then prove  the conjecture of Bai, Bai, Guo and Wu in \cite{BBGW} under certain strong condition.
	\end{abstract}
	
	\maketitle
	
	\section{Introduction}
			  A Poisson algebra is a triple
		$(L, \cdot ,\left[ { - , - } \right])$, where $(L, \cdot )$ is a commutative associative algebra and $(L,\left[ { - , - } \right])$ is a Lie algebra that satisfies the following Leibniz rule:
		$$\left[ {x,y\cdot z} \right] = \left[ {x,y} \right]\cdot z + y\cdot \left[ {x,z} \right],\forall x,y,z \in L.$$

		Poisson algebras  appeared naturally from the study of Hamiltonian mechanics and have played  a significant role
		in mathematics and physics, such as Poisson manifolds,
		integral systems,  algebraic geometry, quantum groups and quantum field theory (see \cite{BV}\cite{CP}\cite{Od}\cite{Po}).
Poisson algebras can be viewed as the algebraic counterpart of Poisson manifolds.
		In the study of Poisson algebras,  many other algebraic structures have been found such as  Jacobi algebras \cite{AM}\cite{CK0}, Poisson bialgebras \cite{LBS}\cite{NB},  Gerstenhaber algebras and Lie-Rinehart algebras \cite{Ger}\cite{KS}\cite{Rin}, $F$-manifold algebras
		\cite{Dot}, Novikov-Poisson algebras \cite{XuX}, quasi-Poisson algebras \cite{BY}
		and Poisson $n$-Lie algebras \cite{CK}.
		
		As a dual
		notion of a Poisson algebra, the concept of a transposed Poisson algebra has  recently been  introduced by Bai, Bai, Guo and Wu  \cite{BBGW}.
		A transposed Poisson algebra $(L, \cdot ,\left[ { - , - } \right])$ is defined by exchanging the roles of the two binary operations in the Leibniz rule
		defining the Poisson algebra:
		$$2z \cdot \left[ {x,y} \right] = \left[ {z \cdot x,y} \right] + \left[ {x,z \cdot y} \right],\forall x,y,z \in L,$$
		where $(L, \cdot )$ is a commutative associative algebra and $(L,\left[ { - , - } \right])$ is a Lie algebra.
		
		It is shown that a transposed Poisson algebra  possesses many important identities and properties, and can be naturally obtained   through taking the
		commutator in the Novikov-Poisson algebra \cite{BBGW}. There are many results on transposed Poisson algebras, such as transposed
		Hom-Poisson algebras \cite{LS}, transposed BiHom-Poisson algebras \cite{ML}, a bialgebra theory for transposed Poisson algebras \cite{LB}, a relation between $\frac{1}{2}$-derivations of Lie algebras and transposed
		Poisson algebras \cite{FKL}, and the tranposed Poisson structures with fixed Lie
		algebras (see \cite{BFK} for more details).
		
		The notion of  an $n$-Lie algebra (see Definition \ref{def2.2}) introduced by Filippov \cite{Fil} has found use in many fields in mathematics
		and physics \cite{BN1}\cite{BN2}\cite{Nam}\cite{Tak}. To explicitly construct $n$-Lie algebras becomes one of the important problems in this theory. In \cite{Dzh}, Dzhumadil$'$dav introduced the notion of a Poisson $n$-Lie algebra which can be used to construct an ($n+1$)-Lie algebra under  an additional strong condition. In \cite{BBGW}, Bai, Bai, Guo and Wu showed that this strong condition for $n=2$
		holds automatically for a transposed Poisson algebra, and gave a construction of 3-Lie algebras from transposed Poisson algebras with derivations.  They also found that this constructed $3$-Lie algebra and the commutative
		associative algebra satisfy the   analog of the compatibility condition in the transposed Poisson  algebra, which is called a transposed Poisson $3$-Lie algebra.
		This motivates them to introduce the concept of a transposed Poisson $n$-Lie algebra (see Definition \ref{def2.3}), and then to propose the following conjecture:
		\begin{conjecture}\cite{BBGW}\label{con}
			Let $n\geq 2$ be an integer and $(L,\cdot,\mu_n)$  a transposed Poisson $n$-Lie algebra. Let $D$ be a derivation of $(L,\cdot)$ and $(L,\mu_n)$.
			Define an
			$(n+1)$-ary operation
			$$\mu_{n+1}(x_1,\cdots,x_{n+1}):=\sum_{i=1}^{n+1} (-1)^{i-1}D(x_i) \mu_n(x_1,\cdots, \hat{x}_i,\cdots,x_{n+1}),\quad \forall x_1,\cdots,x_{n+1}\in L,
			$$
			where $\hat{x}_i$ means that the $i$-th entry is omitted. Then $(L,\cdot,\mu_{n+1})$ is a transposed Poisson $(n+1)$-Lie algebra.
		\end{conjecture}

In this  paper, based on the identities for transposed Poisson $n$-Lie algebras given in Section 2, we  prove that Conjecture  \ref{con} holds under some strong condition described in Section 3 (see Definition \ref{def2.4} and Theorem \ref{main}).
		
		Throughout the paper, all vector spaces are taken over a field  of characteristic zero. To simplify notations, the commutative associative multiplication $ \cdot $ will  be omitted unless the emphasis is needed.

	\section{Identities in  transposed Poisson $n$-Lie algebras}
	
In this section, we firstly recall some definitions, and then construct a class of  identities for transposed Poisson $n$-Lie algebras.

\begin{definition}\cite{Fil}\label{def2.2}
Let $n\geq 2$ be an integer. An $n$-Lie algebra is a vector space $L$ together with a skew-symmetric linear map $\left[ { - , \cdots , - } \right]:{ \otimes ^n}L \to L$ such that for any $x_i, y_j \in L, 1 \le i \le n-1, 1 \le j \le n,$ the following identity holds:
\begin{equation}\label{(nl)}
		{\left[ {{{\left[ {{y_1}, \cdots ,{y_n}} \right]}},{x_1}, \cdots ,{x_{n - 1}}} \right]} = \sum\limits_{i = 1}^n {{{( - 1)}^{i - 1}}{{\left[ {{{\left[ {{y_i},{x_1}, \cdots ,{x_{n - 1}}} \right]}},{y_1}, \cdots ,{{\hat y}_i}, \cdots ,{y_n}} \right]}}}.
	\end{equation}
\end{definition}

\begin{definition}\cite{BBGW}\label{def2.3}
Let $n\geq 2$ be an integer and $L$  a vector space. The triple $(L, \cdot ,\left[ { - , \cdots , - } \right])$ is called a  transposed Poisson $n$-Lie algebra if $(L, \cdot )$ is a commutative associative algebra and $(L,\left[ { - , \cdots , - } \right])$ is an $n$-Lie algebra such that for any $h,{x_i} \in L, 1 \le i \le n$,  the following identity holds:	
	\begin{equation}\label{(np)}
		nh\left[ {{x_1}, \cdots ,{x_n}} \right] = \sum\limits_{i = 1}^n {\left[ {{x_1}, \cdots ,h{x_i}, \cdots ,{x_n}} \right]}.
	\end{equation}	
\end{definition}

Some identities for transposed Poisson  algebras in \cite{BBGW} can be extended to the following for transposed Poisson $n$-Lie algebras.

\begin{theorem}\label{2.4}
	Let  $(L, \cdot ,\left[ { - , \cdots , - } \right])$ be a transposed Poisson $n$-Lie algebra. Then the following identities
	hold:
\begin{itemize}
\item[(1)] For any  $x_i\in L, 1\leq i\leq n+1$, we have
\begin{align}\label{(np1)}
\sum\limits_{i = 1}^{n + 1} {{{( - 1)}^{i - 1}}{x_i}\left[ {{x_1}, \cdots ,{{\hat x}_i}, \cdots ,{x_{n + 1}}} \right]}  = 0;
\end{align}
\item[(2)] For any $h\in L$,  $x_i\in L, 1\leq i\leq n-1$ and $y_i\in L, 1\leq i\leq n$, we have
\begin{align}\label{(np2)}
\sum\limits_{i = 1}^n {{{( - 1)}^{i - 1}}\left[ {h\left[ {{y_i},{x_1}, \cdots ,{x_{n - 1}}} \right],{y_1}, \cdots ,{{\hat y}_i}, \cdots ,{y_n}} \right] = } \left[ {h\left[ {{y_1}, \cdots ,{y_n}} \right],{x_1}, \cdots ,{x_{n - 1}}} \right];
\end{align}
\item[(3)] For any  $x_i\in L, 1\leq i\leq n-1$ and $y_i\in L, 1\leq i\leq n+1$, we have
\begin{align}\label{(np3)}
\sum\limits_{i = 1}^{n + 1} {{{( - 1)}^{i - 1}}\left[{y_i}, {{x_1}, \cdots ,{x_{n-1}}} \right]\left[ {{y_1}, \cdots ,{{\hat y}_i}, \cdots ,{y_{n + 1}}} \right]}  = 0;
\end{align}
\item[(4)] For any  $x_1,x_2\in L$ and $y_i\in L, 1\leq i\leq n$, we have
\begin{align}\label{(np4)}
\sum\limits_{i = 1}^n {\sum\limits_{j = 1,j \ne i}^n {\left[ {{y_1}, \cdots ,{y_i}{x_1}, \cdots ,{y_j}{x_2}, \cdots ,{y_n}} \right]} }  = n(n - 1){x_1}{x_2}\left[ {{y_1},{y_2}, \cdots ,{y_n}} \right].
\end{align}
\end{itemize}
\end{theorem}

\begin{proof}
(1)  By Equation (\ref{(np)}), for any $1\leq i\leq n+1$, we have
	$$n{x_i}\left[ {{x_1}, \cdots ,{x_{i - 1}},{x_{i + 1}}, \cdots ,{x_{n+1}}} \right] = \sum\limits_{j \ne i}^{} {\left[ {{x_1}, \cdots ,{x_{i - 1}},{x_{i + 1}}, \cdots ,{x_i}{x_j}, \cdots ,{x_{n+1}}} \right]}. $$
	
Thus, we obtain
	$$\sum\limits_{i = 1}^{n + 1} {{{( - 1)}^{i - 1}}n{x_i}\left[ {{x_1}, \cdots ,{{\hat x}_i}, \cdots ,{x_{n + 1}}} \right]}  = \sum\limits_{i = 1}^{n + 1} {\sum\limits_{j = 1,j \ne i}^{n + 1} {{{( - 1)}^{i - 1}}\left[ {{x_1}, \cdots ,{{\hat x}_i}, \cdots ,{x_i}{x_j}, \cdots ,{x_{n + 1}}} \right]} }.$$
	
Note that for any $i>j$, we have
\begin{eqnarray*}
&&{( - 1)^{i - 1}}\left[ {{x_1}, \cdots ,{x_{j - 1}},{x_i}{x_j},{x_{j + 1}}, \cdots ,{{\hat x}_i}, \cdots ,{x_n}} \right] \\
&&+ {( - 1)^{j - 1}}\left[ {{x_1}, \cdots ,{{\hat x}_j}, \cdots ,{x_{i - 1}},{x_j}{x_i},{x_{i + 1}}, \cdots ,{x_n}} \right]\\
&=& {( - 1)^{i - 1 + (i - j - 1)}}\left[ {{x_1}, \cdots ,{x_{j - 1}},{x_{j + 1}}, \cdots ,{x_{i - 1}},{x_i}{x_j},{x_{i + 1}}, \cdots ,{x_n}} \right] \\
&&	+ {( - 1)^{j - 1}}\left[ {{x_1}, \cdots ,{{\hat x}_j}, \cdots ,{x_{i - 1}},{x_j}{x_i},{x_{i + 1}}, \cdots ,{x_n}} \right]\\
&=& \big( {{{( - 1)}^{ - j - 2}} + {{( - 1)}^{j - 1}}} \big)\left[ {{x_1}, \cdots ,{x_{j - 1}},{x_{j + 1}}, \cdots ,{x_{i - 1}},{x_i}{x_j},{x_{i + 1}}, \cdots ,{x_n}} \right]\\
&=& 0,
\end{eqnarray*}	
which gives
$\sum\limits_{i = 1}^{n + 1} {\sum\limits_{j = 1,j \ne i}^{n + 1} {{{( - 1)}^{i - 1}}\left[ {{x_1}, \cdots ,{{\hat x}_i}, \cdots ,{x_i}{x_j}, \cdots ,{x_{n + 1}}} \right]} } =0.$
	
Hence, we get
	$$\sum\limits_{i = 1}^{n + 1} {{{( - 1)}^{i - 1}}n{x_i}\left[ {{x_1}, \cdots ,{{\hat x}_i}, \cdots ,{x_{n + 1}}} \right]}  =0.$$
	
(2)  By Equation (\ref{(np)}), we have
\begin{eqnarray*}
	&&- \left[ {h\left[ {{y_1}, \cdots ,{y_n}} \right],{x_1}, \cdots ,{x_{n - 1}}} \right] - \sum\limits_{i = 1}^{n - 1} {\left[ {\left[ {{y_1}, \cdots ,{y_n}} \right],{x_1}, \cdots ,h{x_i}, \cdots ,{x_{n - 1}}} \right]} \\
	&= & - nh\left[ {\left[ {{y_1}, \cdots ,{y_n}} \right],{x_1}, \cdots ,{x_{n - 1}}} \right],
	\end{eqnarray*}	
and for any $1\leq j\leq n$,
	\begin{eqnarray*}
	&&{( - 1)^{j - 1}}(\left[ {h\left[ {{y_j},{x_1}, \cdots ,{x_{n - 1}}} \right],{y_1}, \cdots ,{\hat{y}_j}, \cdots ,{y_{n - 1}}} \right] \\
	&&+\sum\limits_{i = 1, i\neq j}^{n } {\left[ {\left[ {{y_j},{x_1}, \cdots ,{x_{n - 1}}} \right],{y_1}, \cdots ,h{y_i}, \cdots ,\hat{y}_j, \cdots ,{y_{n - 1}}} \right]}) \\
	&= &{( - 1)^{j - 1}}nh\left[ {\left[ {{y_j},{x_1}, \cdots ,{x_{n - 1}}} \right],{y_1}, \cdots ,{\hat{y}_j}, \cdots ,{y_{n - 1}}} \right].
	\end{eqnarray*}

By taking the sum of the above  $n+1$ identities and applying Equation (\ref{(nl)}), we get
\begin{eqnarray*}
		&&- \left[ {h\left[ {{y_1}, \cdots ,{y_n}} \right],{x_1}, \cdots ,{x_{n - 1}}} \right]- \sum\limits_{i = 1}^{n - 1} {\left[ {\left[ {{y_1}, \cdots ,{y_n}} \right],{x_1}, \cdots ,h{x_i}, \cdots ,{x_{n - 1}}} \right]}\\
&&+\sum\limits_{j = 1}^{n} {( - 1)^{j - 1}}(\left[ {h\left[ {{y_j},{x_1}, \cdots ,{x_{n - 1}}} \right],{y_1}, \cdots ,{\hat{y}_j}, \cdots ,{y_{n - 1}}} \right] \\
	&&+\sum\limits_{i = 1, i\neq j}^{n } {\left[ {\left[ {{y_j},{x_1}, \cdots ,{x_{n - 1}}} \right],{y_1}, \cdots ,h{y_i}, \cdots ,\hat{y}_j, \cdots ,{y_{n - 1}}} \right]}) \\
&= &  -nh\left[ {\left[ {{y_1}, \cdots ,{y_n}} \right],{x_1}, \cdots ,{x_{n - 1}}} \right]+\\
&&nh\sum\limits_{j = 1}^{n} {( - 1)^{j - 1}}\left[ {\left[ {{y_j},{x_1}, \cdots ,{x_{n - 1}}} \right],{y_1}, \cdots ,{\hat{y}_j}, \cdots ,{y_{n - 1}}} \right]\\
&=&   0.
	\end{eqnarray*}
	
We denote
\begin{eqnarray*}
 {{A_j}} & := &\sum\limits_{i = 1,i \ne j}^n {{{( - 1)}^{i - 1}}\left[ {\left[ {{y_i},{x_1}, \cdots ,{x_{n - 1}}} \right],{y_1}, \cdots ,h{y_j}, \cdots ,{{\hat y}_i}, \cdots ,{y_n}} \right]}, 1\leq j\leq {n},\\
	{B_i} &:=& \left[ {\left[ {{y_1}, \cdots ,{y_n}} \right],{x_1}, \cdots ,h{x_i}, \cdots ,{x_{n - 1}}} \right], 1\leq i\leq {n-1}.
\end{eqnarray*}
Then the above equation can be rewritten as
\begin{equation}\label{(np2.1)}
	\begin{array}{l}
		\sum\limits_{i = 1}^n {{{( - 1)}^{i - 1}}\left[ {h\left[ {{y_i},{x_1}, \cdots ,{x_{n - 1}}} \right],{y_1}, \cdots ,{{\hat y}_i}, \cdots ,{y_n}} \right]}  - \left[ {h\left[ {{y_1}, \cdots ,{y_n}} \right],{x_1}, \cdots ,{x_{n - 1}}} \right]\\
		+ \sum\limits_{j = 1}^n {{A_j}}  - \sum\limits_{i = 1}^{n - 1} {{B_i}}  = 0.
	\end{array}
\end{equation}

By applying Equation (\ref{(nl)}) to $ {{A_j}}, 1\leq j\leq n  $, we have	
	\begin{eqnarray*}
		{{A_j}} & =& \sum\limits_{i = 1,i \ne j}^n {{{( - 1)}^{i - 1}}\left[ {\left[ {{y_i},{x_1}, \cdots ,{x_{n - 1}}} \right],{y_1}, \cdots ,h{y_j}, \cdots ,{{\hat y}_i}, \cdots ,{y_n}} \right]}\\
		&=& \left[ {\left[ {{y_1}, \cdots ,h{y_j}, \cdots ,{y_n}} \right],{x_1}, \cdots ,{x_{n - 1}}} \right] \\
		&&+ {( - 1)^j}\left[ {\left[ {h{y_j},{x_1}, \cdots ,{x_{n - 1}}} \right],{y_1}, \cdots ,{{\hat y}_j}, \cdots ,{y_n}} \right].
	\end{eqnarray*}
	
Thus, we get
\begin{eqnarray*}
		\sum\limits_{j = 1}^n { {{A_j}} } &=&\sum\limits_{j = 1}^n {\left[ {\left[ {{y_1}, \cdots ,h{y_j}, \cdots ,{y_n}} \right],{x_1}, \cdots ,{x_{n - 1}}} \right]} \\
		&& + \sum\limits_{j = 1}^n {( - 1)^j}{\left[ {\left[ {h{y_j},{x_1}, \cdots ,{x_{n - 1}}} \right],{y_1}, \cdots ,{{\hat y}_j}, \cdots ,{y_n}} \right]}\\
		&=& n\left[ {h\left[ {{y_1},  \cdots ,{y_n}} \right],{x_1}, \cdots ,{x_{n - 1}}} \right]\\
		&& + \sum\limits_{j = 1}^n {( - 1)^j}{\left[ {\left[ {h{y_j},{x_1}, \cdots ,{x_{n - 1}}} \right],{y_1}, \cdots ,{{\hat y}_j}, \cdots ,{y_n}} \right]}.
	\end{eqnarray*}
	
By applying Equation (\ref{(nl)}) to $ {{B_i}}, 1\leq i\leq n-1  $, we have
\begin{eqnarray*}	
&&\left[ {\left[ {{y_1}, \cdots ,{y_n}} \right],{x_1}, \cdots ,h{x_i}, \cdots ,{x_{n - 1}}} \right] \\
&=& \sum\limits_{j = 1}^n {{{( - 1)}^{j - 1}}\left[ {\left[ {{y_j},{x_1}, \cdots ,h{x_i}, \cdots ,{x_{n - 1}}} \right],{y_1}, \cdots {{\hat y}_j}, \cdots ,{y_n}} \right]}.
\end{eqnarray*}	

Thus, we get
	\begin{eqnarray*}
		\sum\limits_{i = 1}^{n - 1} {{B_i}}  &=& \sum\limits_{i = 1}^{n - 1} {\sum\limits_{j = 1}^n {{{( - 1)}^{j - 1}}\left[ {\left[ {{y_j},{x_1}, \cdots ,h{x_i}, \cdots ,{x_{n - 1}}} \right],{y_1}, \cdots {{\hat y}_j}, \cdots ,{y_n}} \right]} } \\
		&= &\sum\limits_{j = 1}^n {\sum\limits_{i = 1}^{n - 1} {{{( - 1)}^{j - 1}}\left[ {\left[ {{y_j},{x_1}, \cdots ,h{x_i}, \cdots ,{x_{n - 1}}} \right],{y_1}, \cdots {{\hat y}_j}, \cdots ,{y_n}} \right]} }.
	\end{eqnarray*}
	
	Note that, by Equation (\ref{(np)}), we have	
	\begin{eqnarray*}
		&&\sum\limits_{i = 1}^{n - 1} {{{( - 1)}^{j-1}}\left[ {\left[ {{y_j},{x_1}, \cdots ,h{x_i}, \cdots ,{x_{n - 1}}} \right],{y_1}, \cdots {{\hat y}_j}, \cdots ,{y_n}} \right]} \\
	&	= &{( - 1)^{j-1}}n\left[ {h\left[ {{y_j},{x_1}, \cdots ,{x_i}, \cdots ,{x_{n - 1}}} \right],{y_1}, \cdots {{\hat y}_j}, \cdots ,{y_n}} \right] \\
		&&+{( - 1)^{j}} \left[ {\left[ {h{y_j},{x_1}, \cdots ,{x_i}, \cdots ,{x_{n - 1}}} \right],{y_1}, \cdots {{\hat y}_j}, \cdots ,{y_n}} \right].
	\end{eqnarray*}
	
Thus, we obtain
	\begin{eqnarray*}
		\sum\limits_{i = 1}^{n - 1} {{B_i}} &= &\sum\limits_{j = 1}^n {{{( - 1)}^{j - 1}}n\left[ {h\left[ {{y_j},{x_1}, \cdots ,{x_i}, \cdots ,{x_{n - 1}}} \right],{y_1}, \cdots {{\hat y}_j}, \cdots ,{y_n}} \right]} \\
	&&	+ \sum\limits_{j = 1}^n {{{( - 1)}^j}\left[ {\left[ {h{y_j},{x_1}, \cdots ,{x_i}, \cdots ,{x_{n - 1}}} \right],{y_1}, \cdots {{\hat y}_j}, \cdots ,{y_n}} \right]}.
\end{eqnarray*}
	
By substituting these into (\ref{(np2.1)}), we have	
	\begin{eqnarray*}
		&&\sum\limits_{i = 1}^n {{{( - 1)}^{i - 1}}\left[ {h\left[ {{y_i},{x_1}, \cdots ,{x_{n - 1}}} \right],{y_1}, \cdots ,{{\hat y}_i}, \cdots ,{y_n}} \right]}  - \left[ {h\left[ {{y_1}, \cdots ,{y_n}} \right],{x_1}, \cdots ,{x_{n - 1}}} \right]\\
        &&+ n\left[ {h\left[ {{y_1}, \cdots ,{y_n}} \right],{x_1}, \cdots ,{x_{n - 1}}} \right] + \sum\limits_{j = 1}^n {( - 1)^j}{\left[ {\left[ {h{y_j},{x_1}, \cdots ,{x_{n - 1}}} \right],{y_1}, \cdots ,{{\hat y}_j}, \cdots ,{y_n}} \right]}\\
		&&- \sum\limits_{j = 1}^n {{{( - 1)}^{j - 1}}n\left[ {h\left[ {{y_j},{x_1}, \cdots ,{x_i}, \cdots ,{x_{n - 1}}} \right],{y_1}, \cdots {{\hat y}_j}, \cdots ,{y_n}} \right]} \\
	&&-\sum\limits_{j = 1}^n {{{( - 1)}^j}\left[ {\left[ {h{y_j},{x_1}, \cdots ,{x_i}, \cdots ,{x_{n - 1}}} \right],{y_1}, \cdots {{\hat y}_j}, \cdots ,{y_n}} \right]} \\
	&=&0,
	\end{eqnarray*}
 which implies
$$(n - 1)(\sum\limits_{i = 1}^n {{{( - 1)}^i}\left[ {h\left[ {{y_i},{x_1}, \cdots ,{x_{n - 1}}} \right],{y_1}, \cdots ,{{\hat y}_i}, \cdots ,{y_n}} \right]}
	+ \left[ {h\left[ {{y_1}, \cdots ,{y_n}} \right],{x_1}, \cdots ,{x_{n - 1}}} \right])=0.$$
	Therefore the proof of  (\ref{(np2)}) is completed.

(3) By Equation (\ref{(np)}), for any $1\leq j\leq n+1$, we have
\begin{eqnarray*}
	 &&{( - 1)^{j - 1}}n\left[ {{y_j},{x_1}, \cdots ,{x_{n - 1}}} \right]\left[ {{y_1}, \cdots ,{{\hat y}_j}, \cdots ,{y_{n + 1}}} \right]\\
	 &=&	\sum\limits_{i = 1,i \ne j}^{n + 1} {{{( - 1)}^{j - 1}}\left[ {{y_1}, \cdots ,{y_i}\left[ {{y_j},{x_1}, \cdots ,{x_{n - 1}}} \right], \cdots ,{{\hat y}_j}, \cdots ,{y_{n + 1}}} \right]}.
	\end{eqnarray*}	
By taking the sum of the above $n+1$ identities, we obtain
	\begin{eqnarray*}
		&&\sum\limits_{j = 1}^{n + 1} {{{( - 1)}^{j - 1}}n\left[ {{y_j},{x_1}, \cdots ,{x_{n - 1}}} \right]\left[ {{y_1}, \cdots ,{{\hat  y}_j}, \cdots ,{y_{n + 1}}} \right]} \\
		&=& \sum\limits_{j = 1}^{n + 1} {\sum\limits_{i = 1,i \ne j}^{n + 1} {{{( - 1)}^{j - 1}}\left[ {{y_1}, \cdots ,{y_i}\left[ {{y_j},{x_1}, \cdots ,{x_{n - 1}}} \right], \cdots ,{{\hat y}_j}, \cdots ,{y_{n + 1}}} \right]} }.
	\end{eqnarray*}
	
Thus we only need to  prove the following equation:
	$$	\sum\limits_{j = 1}^{n + 1} {\sum\limits_{i = 1,i \ne j}^{n + 1} {{{( - 1)}^{j - 1}}\left[ {{y_1}, \cdots ,{y_i}\left[ {{y_j},{x_1}, \cdots ,{x_{n - 1}}} \right], \cdots ,{{\hat y}_j}, \cdots ,{y_{n + 1}}} \right]} } =0.$$
 Note that
	\begin{eqnarray*}
		&&\sum\limits_{j = 1}^{n + 1} {\sum\limits_{i = 1,i \ne j}^{n + 1} {{{( - 1)}^{j - 1}}\left[ {{y_1}, \cdots ,{y_i}\left[ {{y_j},{x_1}, \cdots ,{x_{n - 1}}} \right], \cdots ,{{\hat y}_j}, \cdots ,{y_{n + 1}}} \right]} } \\
		&	= &\sum\limits_{i = 1}^{n + 1} {\sum\limits_{j = 1,j \ne i}^{n + 1} {{{( - 1)}^{j - 1}}\left[ {{y_1}, \cdots ,{y_i}\left[ {{y_j},{x_1}, \cdots ,{x_{n - 1}}} \right], \cdots ,{{\hat y}_j}, \cdots ,{y_{n + 1}}} \right]} } \\
			&=& \sum\limits_{i = 1}^{n + 1} {\sum\limits_{j = 1}^{i - 1} {{{( - 1)}^{i + j - 1}}\left[ {{y_i}\left[ {{y_j},{x_1}, \cdots ,{x_{n - 1}}} \right],{y_1}, \cdots ,{{\hat y}_j}, \cdots ,{{\hat y}_i}, \cdots ,{y_{n + 1}}} \right]} } \\
		&&	+ \sum\limits_{i = 1}^{n + 1} {\sum\limits_{j = i + 1}^{n + 1} {{{( - 1)}^{i + j}}\left[ {{y_i}\left[ {{y_j},{x_1}, \cdots ,{x_{n - 1}}} \right],{y_1}, \cdots ,{{\hat y}_i}, \cdots ,{{\hat y}_j}, \cdots ,{y_{n + 1}}} \right]} } \\
       &\mathop =  \limits^{(\ref{(np2)})}& \sum\limits_{i = 1}^{n + 1} {{{( - 1)}^i}\left[ {{y_i}\left[ {{y_1}, \cdots ,{{\hat y}_i}, \cdots ,{y_{n + 1}}} \right],{x_1}, \cdots ,{x_{n - 1}}} \right]} \\
	&\mathop  = \limits^{(\ref{(np1)})}& 0.
\end{eqnarray*}	
	Hence the conclusion holds.	
	
(4) By applying Equation (\ref{def2.3}), we have
	\begin{eqnarray*}
		&&{n^2}{x_1}{x_2}\left[ {{y_1},{y_2}, \cdots ,{y_n}} \right] = n{x_1}\sum\limits_{j = 1}^n {\left[ {{y_1}, \cdots ,{y_j}{x_2}, \cdots ,{y_n}} \right]} \\
	&=& \sum\limits_{i = 1}^n {\sum\limits_{j = 1,j \ne i}^n {\left[ {{y_1}, \cdots ,{y_i}{x_1}, \cdots ,{y_j}{x_2}, \cdots ,{y_n}} \right]} }  + \sum\limits_{j = 1}^n {\left[ {{y_1}, \cdots ,{y_j}{x_1}{x_2}, \cdots ,{y_n}} \right]}\\
	&= &\sum\limits_{i = 1}^n {\sum\limits_{j = 1,j \ne i}^n {\left[ {{y_1}, \cdots ,{y_i}{x_1}, \cdots ,{y_j}{x_2}, \cdots ,{y_n}} \right]} }  + n{x_1}{x_2}\left[ {{y_1}, \cdots ,{y_n}} \right],
\end{eqnarray*}
which gives
$$n(n - 1){x_1}{x_2}\left[ {{y_1},{y_2}, \cdots ,{y_n}} \right] = \sum\limits_{i = 1}^n {\sum\limits_{j = 1,j \ne i}^n {\left[ {{y_1}, \cdots ,{y_i}{x_1}, \cdots ,{y_j}{x_2}, \cdots ,{y_n}} \right]} } .$$
Hence the proof is completed.	
\end{proof}

To prove Conjecture \ref{con}, we need the following extra condition.
\begin{definition}\label{def2.4}
	A transposed Poisson $n$-Lie algebra $(L, \cdot ,\left[ { - , \cdots , - } \right])$ is called  strong if the following identity holds	
	\begin{equation}\label{(strong)}
		\begin{array}{l}
		{y_1}\left[ {h{y_2},{x_1}, \cdots ,{x_{n - 1}}} \right] - {y_2}\left[ {h{y_1},{x_1}, \cdots ,{x_{n - 1}}} \right] + \sum\limits_{i = 1}^{n - 1} {{{( - 1)}^{i - 1}}h{x_i}\left[ {{y_1},{y_2},{x_1}, \cdots ,{{\hat x}_i}, \cdots ,{x_{n - 1}}} \right]} \\
		 = 0
	\end{array}
	\end{equation}
for any  ${y_1},{y_2}, {x_i}\in L, 1\leq i\leq n - 1.$
\end{definition}
\begin{remark}
 When $n=2$, the identity is $${y_1}\left[ {h{y_2},{x_1}} \right] + {y_2}\left[ {{x_1},h{y_1}} \right] + h{x_1}\left[ {{y_1},{y_2}} \right] = 0$$ which is exactly Theorem 2.5 (11) in \cite{BBGW}. Thus, in the case of a transposed Poisson  algebra, the strong condition holds naturally.
\end{remark}

\begin{proposition}\label{2.8}
	Let $(L, \cdot ,\left[ { - , \cdots , - } \right])$ be a strong transposed Poisson $n$-Lie algebra. Then	
	\begin{equation}\label{(strong1)}
		{y_1}\left[ {h{y_2},{x_1}, \cdots ,{x_{n - 1}}} \right] - h{y_1}\left[ {{y_2},{x_1}, \cdots ,{x_{n - 1}}} \right] = {y_2}\left[ {h{y_1},{x_1}, \cdots ,{x_{n - 1}}} \right] - h{y_2}\left[ {{y_1},{x_1}, \cdots ,{x_{n - 1}}} \right]
	\end{equation}	
for any  ${y_1},{y_2}, {x_i} \in L, 1\leq i\leq n - 1$.	
\end{proposition}
\begin{proof}
By Equation (\ref{(np1)}), we have
$$ - h{y_1}\left[ {{y_2},{x_1}, \cdots ,{x_{n - 1}}} \right] + h{y_2}\left[ {{y_1},{x_1}, \cdots ,{x_{n - 1}}} \right] = \sum\limits_{i = 1}^{n - 1} {{{( - 1)}^{i - 1}}h{x_i}\left[ {{y_1},{y_2},{x_1}, \cdots ,{{\hat x}_i}, \cdots ,{x_{n - 1}}} \right]}.$$
Then the statement follows from  Equation (\ref{(strong)}).
\end{proof}

\section{Proof of the conjecture for strong transposed Poisson $n$-Lie algebras}
	
In this section, we will prove Conjecture \ref{con} for strong transposed Poisson $n$-Lie algebras. Fist we recall the  notion of derivations of  transposed Poisson $n$-Lie algebras.
	
\begin{definition}\label{der}
Let $(L, \cdot, \left[ { - , \cdots , - } \right])$ be a transposed Poisson $n$-Lie algebra. The linear operation $D$ is call a derivation of $(L, \cdot, \left[ { - , \cdots , - } \right])$, if
 the following holds for any $u,v, x_{i}\in L, 1\leq i\leq n$:	
 \begin{itemize}
\item[(1)] $D$ is a derivation of $\left({L, \cdot} \right)$ ,i.e., $D\left( {uv} \right) = D\left( u \right)v + uD\left( v \right)$;
\item[(2)] $D$ is a derivation of $(L, \left[ { - , \cdots , - } \right])$, i.e.,
		$$D\left( \left[ {{x_1}, \cdots ,x_n}  \right] \right)= \sum\limits_{i = 1}^n {\left[ {{x_1}, \cdots ,{x_{i - 1}},D\left( {{x_i}} \right),{x_{i + 1}}, \cdots ,{x_n}} \right]}.$$
\end{itemize}
\end{definition}

\begin{lemma}\label{3.3}
	Let $(L, \cdot ,\left[ { - , \cdots , - } \right] )$ be a  transposed Poisson $n$-Lie algebra and $D$ a derivation of $(L, \cdot ,\left[ { - , \cdots , - } \right] )$. For any  $y_i\in L, 1\leq i\leq n+1$, we have
\begin{itemize}
\item[(1)] \begin{align} \label{(3.1)}
 &\sum\limits_{i = 1}^{n + 1} {{{( - 1)}^{i - 1}}D({y_i})} D\left( {\left[ {{y_1}, \cdots ,{{\hat y}_i}, \cdots ,{y_{n + 1}}} \right]} \right)\nonumber\\
			=&\sum\limits_{i = 1}^{n + 1} {\sum\limits_{j = 1,j \ne i}^{n + 1} {{{( - 1)}^{i - 1}}D({y_i})\left[ {{y_1}, \cdots ,D({y_j}), \cdots ,{{\hat y}_i}, \cdots ,{y_{n + 1}}} \right]} };
\end{align}
\item[(2)]
\begin{align}\label{(3.2)}
 &\sum\limits_{i = 1}^{n + 1} {{{( - 1)}^{i - 1}}D({y_i})} D\left( {\left[ {{y_1}, \cdots ,{{\hat y}_i}, \cdots ,{y_{n + 1}}} \right]} \right)\nonumber\\
			= &\sum\limits_{i = 1}^{n + 1} {\sum\limits_{j = 1,j \ne i}^{n + 1} {\sum\limits_{k = j + 1,k \ne i}^{n + 1} {{{( - 1)}^i}{y_i}\left[ {{y_1}, \cdots ,D({y_j}), \cdots ,D({y_k}), \cdots ,{{\hat y}_i}, \cdots ,{y_{n + 1}}} \right]} } },
\end{align}	
\end{itemize}
where for any  $i>j$, $\sum\limits_{i}^{j} $ means the empty sum which is zero.	
\end{lemma}
\begin{proof}
(1) The statement follows immediately from Definition \ref{der}.

(2) By applying Equation (\ref{(3.1)}),	we need to prove the following equation
	\begin{eqnarray*}
	&	&\sum\limits_{i = 1}^{n + 1} {\sum\limits_{j = 1,j \ne i}^{n + 1} {{{( - 1)}^{i - 1}}nD({y_i})\left[ {{y_1}, \cdots ,D({y_j}), \cdots ,{{\hat y}_i}, \cdots ,{y_{n + 1}}} \right]} } \\
	&	=& \sum\limits_{i = 1}^{n + 1} {\sum\limits_{j = 1,j \ne i}^{n + 1} {\sum\limits_{k = j + 1,k \ne i}^{n + 1} {{{( - 1)}^i}n{y_i}\left[ {{y_1}, \cdots ,D({y_j}), \cdots ,D({y_k}), \cdots ,{{\hat y}_i}, \cdots ,{y_{n + 1}}} \right]} } }.
	\end{eqnarray*}

For any $1\leq i\leq n+1$, denote $A_i := n\sum\limits_{j = 1,j \ne i}^{n + 1} {{{( - 1)}^{i - 1}}D({y_i})\left[ {{y_1}, \cdots ,D({y_j}), \cdots ,{{\hat y}_i}, \cdots ,{y_{n + 1}}} \right]}$.
Then we have
\[\sum\limits_{i = 1}^{n + 1} {\sum\limits_{j = 1,j \ne i}^{n + 1} {{{( - 1)}^{i - 1}}nD({y_i})\left[ {{y_1}, \cdots ,D({y_j}), \cdots ,{{\hat y}_i}, \cdots ,{y_{n + 1}}} \right]} }  = \sum\limits_{i = 1}^{n + 1} { {{A_i}} }.\]	
Note that
	\begin{eqnarray*}
	&&{A_i}
	=n\sum\limits_{j = 1,j \ne i}^{n + 1} {{{( - 1)}^{i - 1}}D({y_i})\left[ {{y_1}, \cdots ,D({y_j}), \cdots ,{{\hat y}_i}, \cdots ,{y_{n + 1}}} \right]} \\
	&=& {( - 1)^{i - 1}}(nD({y_i})\left[ {D({y_1}),{y_2}, \cdots ,{{\hat y}_i}, \cdots ,{y_{n + 1}}} \right] + nD({y_i})\left[ {{y_1},D({y_2}),{y_3}, \cdots ,{{\hat y}_i}, \cdots ,{y_{n + 1}}} \right]\\
	&&+  \cdots  + nD({y_i})\left[ {{y_1}, \cdots ,{{\hat y}_i}, \cdots ,{y_n},D({y_{n + 1}})} \right])\\
	&=& {( - 1)^{i - 1}}(\left[ {D({y_i})D({y_1}),{y_2}, \cdots ,{{\hat y}_i}, \cdots ,{y_{n + 1}}} \right] \\
	&&+ \sum\limits_{k = 2,k \ne i}^{n + 1} {\left[ {D({y_1}),{y_2}, \cdots ,{y_k}D({y_i}), \cdots ,{{\hat y}_i}, \cdots ,{y_{n + 1}}} \right]} \\
	&&+ \left[ {{y_1},D({y_i})D({y_2}),{y_3}, \cdots ,{{\hat y}_i}, \cdots ,{y_{n + 1}}} \right]\\
	&& + \sum\limits_{k = 1,k \ne 2,i}^{n + 1} {\left[ {{y_1},D({y_2}),{y_3}, \cdots ,{y_k}D({y_i}), \cdots ,{{\hat y}_i}, \cdots ,{y_{n + 1}}} \right]} \\
	&&+  \cdots  + \left[ {{y_1}, \cdots ,{{\hat y}_i}, \cdots ,{y_n},D({y_i})D({y_{n + 1}})} \right] \\
	&&+ \sum\limits_{k = 1,k \ne i}^n {\left[ {{y_1}, \cdots ,{y_k}D({y_i}), \cdots ,{{\hat y}_i}, \cdots ,{y_n},D({y_{n + 1}})} \right]} )\\
	&=& {( - 1)^{i - 1}}\sum\limits_{j = 1,j \ne i}^{n + 1} {\left[ {{y_1}, \cdots ,D({y_i})D({y_j}), \cdots ,{{\hat y}_i}, \cdots ,{y_{n + 1}}} \right]} \\
	&&+ {( - 1)^{i - 1}}\sum\limits_{j = 1,j \ne i}^{n + 1} {\sum\limits_{k = 1,k \ne j,i}^{n + 1} {\left[ {{y_1}, \cdots ,D({y_j}), \cdots ,{y_k}D({y_i}), \cdots ,{{\hat y}_i}, \cdots ,{y_{n + 1}}} \right]} }.
	\end{eqnarray*}
	
Thus we  have	
	\begin{eqnarray*}
		\sum\limits_{i = 1}^{n + 1} { {{A_i}} } &= &\sum\limits_{j = 1}^{n + 1} {\sum\limits_{i = 1,i \ne j}^{n + 1} {{{( - 1)}^{j - 1}}\left[ {{y_1}, \cdots ,D({y_j})D({y_i}), \cdots ,{{\hat y}_j}, \cdots ,{y_{n + 1}}} \right]} } \\
		&&+ \sum\limits_{i = 1}^{n + 1} {\sum\limits_{j = 1,j \ne i}^{n + 1} {\sum\limits_{k = 1,k \ne i,j}^{n + 1} {{( - 1)^{i-1}}\left[ {{y_1}, \cdots ,D({y_j}), \cdots ,{y_k}D({y_i}), \cdots ,{{\hat y}_i}, \cdots ,{y_{n + 1}}} \right]} } } \\
		&=& {T_1} + {T_2},
	\end{eqnarray*}
where
\begin{eqnarray*}
		{T_1} &:=&\sum\limits_{j = 1}^{n + 1} {\sum\limits_{i = 1,i \ne j}^{n + 1} {{{( - 1)}^{j - 1}}\left[ {{y_1}, \cdots ,D({y_j})D({y_i}), \cdots ,{{\hat y}_j}, \cdots ,{y_{n + 1}}} \right]} },\\		
		{T_2}& :=&  \sum\limits_{i = 1}^{n + 1} {\sum\limits_{j = 1,j \ne i}^{n + 1} {\sum\limits_{k = 1,k \ne i,j}^{n + 1} {{( - 1)^{i-1}}\left[ {{y_1}, \cdots ,D({y_j}), \cdots ,{y_k}D({y_i}), \cdots ,{{\hat y}_i}, \cdots ,{y_{n + 1}}} \right]} } }. \\
		\end{eqnarray*}	
Note that
	\begin{eqnarray*}
		{T_1} =   \sum\limits_{j,i = 1}^{n + 1} {{B_{ji}}} ,
		\end{eqnarray*}		
where ${B_{ji}} = {( - 1)^{j - 1}}\left[ {{y_1}, \cdots ,D({y_j})D({y_i}), \cdots ,{{\hat y}_j}, \cdots ,{y_{n + 1}}} \right]$ for any $1\leq j\neq i\leq n+1$,  and ${B_{ii}} = 0$
for any $1\leq i\leq n+1$.

For any $1\leq i,j\leq n+1,$ without loss of generality, assume $i < j$, then we have
	\begin{eqnarray*}
	&&{B_{ji}} + {B_{ij}}\\
	&=& {( - 1)^{j - 1}}\left[ {{y_1}, \cdots ,D({y_j})D({y_i}), \cdots ,{{\hat y}_j}, \cdots ,{y_{n + 1}}} \right]\\
	&&+ {( - 1)^{i - 1}}\left[ {{y_1}, \cdots ,{{\hat y}_i}, \cdots ,D({y_i})D({y_j}), \cdots ,{y_{n + 1}}} \right]\\
	&=& {( - 1)^{j - 1}}\left[ {{y_1}, \cdots ,D({y_j})D({y_i}), \cdots ,{{\hat y}_j}, \cdots ,{y_{n + 1}}} \right]\\
	&&+ {( - 1)^{i - 1 + j - i + 1}}\left[ {{y_1}, \cdots ,D({y_j})D({y_i}), \cdots ,{{\hat y}_j}, \cdots ,{y_{n + 1}}} \right]\\
	&=&  0,
	\end{eqnarray*}		
which implies ${T_1} = \sum\limits_{j,i = 1}^{n + 1} {{B_{ji}}}  = 0.$
	
Thus we get
		\begin{eqnarray*}
		 & &\sum\limits_{i = 1}^{n + 1} { {{A_i}} } = {T_2} \\
	    &= &\sum\limits_{i = 1}^{n + 1} {\sum\limits_{j = 1,j \ne i}^{n + 1} {\sum\limits_{k = 1,k \ne i,j}^{n + 1} {{( - 1)^{i-1}}\left[ {{y_1}, \cdots ,D({y_j}), \cdots ,{y_k}D({y_i}), \cdots ,{{\hat y}_i}, \cdots ,{y_{n + 1}}} \right]} } }.
\end{eqnarray*}

We rewrite
\begin{eqnarray*}
			&&\sum\limits_{i = 1}^{n + 1} {\sum\limits_{j = 1,j \ne i}^{n+1} {\sum\limits_{k = j + 1,k \ne i}^{n + 1} {{{( - 1)}^i}n{y_i}\left[ {{y_1}, \cdots ,D({y_j}), \cdots ,D({y_k}), \cdots ,{{\hat y}_i}, \cdots ,{y_{n + 1}}} \right]} } }\\
		&=& \sum\limits_{i = 1}^{n + 1} {\sum\limits_{j = 1,j \ne i}^{n+1} {\sum\limits_{k = j + 1,k \ne i}^{n + 1} {\sum\limits_{t = 1,t \ne j,k,i}^{n + 1} {{{( - 1)}^i}}} } }\\
			&&\cdot\left[ {{y_1}, \cdots ,D({y_j}), \cdots ,D({y_k}), \cdots ,{y_t}{y_i}, \cdots ,{{\hat y}_i}, \cdots ,{y_{n + 1}}} \right]  \\
		&&+ \sum\limits_{i = 1}^{n + 1} {\sum\limits_{j = 1,j \ne i}^{n+1} {\sum\limits_{k = j + 1,k \ne i}^{n + 1} {{{( - 1)}^i}\left[ {{y_1}, \cdots ,{y_i}D({y_j}), \cdots ,D({y_k}), \cdots ,{{\hat y}_i}, \cdots ,{y_{n + 1}}} \right]} } } \\
		&&+ \sum\limits_{i = 1}^{n + 1} {\sum\limits_{j = 1,j \ne i}^{n+1} {\sum\limits_{k = j + 1,k \ne i}^{n + 1} {{{( - 1)}^i}\left[ {{y_1}, \cdots ,D({y_j}), \cdots ,{y_i}D({y_k}), \cdots ,{{\hat y}_i}, \cdots ,{y_{n + 1}}} \right]} } } \\
		&=& {M_1} + {M_2} + {M_3},
	\end{eqnarray*}		
where
\begin{eqnarray*}
		{M_1}& :=& \sum\limits_{i = 1}^{n + 1} {\sum\limits_{j = 1,j \ne i}^{n + 1} {\sum\limits_{k = j + 1,k \ne i}^{n + 1} {\sum\limits_{t = 1,t \ne j,k,i}^{n + 1} {{{( - 1)}^i}}} } }\\
		&&\cdot\left[ {{y_1}, \cdots ,D({y_j}), \cdots ,D({y_k}), \cdots ,{y_t}{y_i}, \cdots ,{{\hat y}_i}, \cdots ,{y_{n + 1}}} \right],  \\
		{M_2}& :=& \sum\limits_{i = 1}^{n + 1} {\sum\limits_{j = 1,j \ne i}^{n + 1} {\sum\limits_{k = j + 1,k \ne i}^{n + 1} {{{( - 1)}^i}\left[ {{y_1}, \cdots ,{y_i}D({y_j}), \cdots ,D({y_k}), \cdots ,{{\hat y}_i}, \cdots ,{y_{n + 1}}} \right]} } },\\
		{M_3} &:=& \sum\limits_{i = 1}^{n + 1} {\sum\limits_{j = 1,j \ne i}^{n + 1} {\sum\limits_{k = j + 1,k \ne i}^{n + 1} {{{( - 1)}^i}\left[ {{y_1}, \cdots ,D({y_j}), \cdots ,{y_i}D({y_k}), \cdots ,{{\hat y}_i}, \cdots ,{y_{n + 1}}} \right]} } }.
	\end{eqnarray*}	
	
Note that
\begin{eqnarray*}
		{M_1}& =& \sum\limits_{i = 1}^{n + 1} {\sum\limits_{j = 1,j \ne i}^{n+1} {\sum\limits_{k = j + 1,k \ne i}^{n + 1} {\sum\limits_{t = 1,t \ne j,k,i}^{n + 1} {{{( - 1)}^i}}} } }\\
		&&\cdot\left[ {{y_1}, \cdots ,D({y_j}), \cdots ,D({y_k}), \cdots ,{y_t}{y_i}, \cdots ,{{\hat y}_i}, \cdots ,{y_{n + 1}}} \right]  \\
		& = &\sum\limits_{i,j,k,t = 1}^{n + 1} {{B_{ijkt}}},
	\end{eqnarray*}	
where
\begin{gather*}
	{B_{ijkt}} =
	\begin{cases}
		0,  {\hspace{5.6cm}}\text{if any two indices are equal or } k<j;
		\\
		{( - 1)^i}\left[ {{y_1}, \cdots ,D({y_j}), \cdots ,D({y_k}), \cdots ,{y_t}{y_i}, \cdots ,{{\hat y}_i}, \cdots ,{y_{n + 1}}} \right],{\hspace{0.5cm}}
		\text{otherwise}.
	\end{cases}
\end{gather*}

For any $1\leq j,k\leq n+1,$ without loss of generality, assume $t < i,$   then we have
\begin{eqnarray*}
	&&{B_{ijkt}} + {B_{tjki}}\\
	&=& {( - 1)^i}\left[ {{y_1}, \cdots ,D({y_j}), \cdots ,D({y_k}), \cdots ,{y_t}{y_i}, \cdots ,{{\hat y}_i}, \cdots ,{y_{n + 1}}} \right]\\
	&&+ {( - 1)^t}\left[ {{y_1}, \cdots ,D({y_j}), \cdots ,D({y_k}), \cdots ,{{\hat y}_t}, \cdots ,{y_t}{y_i}, \cdots ,{y_{n + 1}}} \right]\\
	&=& {( - 1)^i}\left[ {{y_1}, \cdots ,D({y_j}), \cdots ,D({y_k}), \cdots ,{y_t}{y_i}, \cdots ,{{\hat y}_i}, \cdots ,{y_{n + 1}}} \right]\\
	&&+ {( - 1)^{t + i - t - 1}}\left[ {{y_1}, \cdots ,D({y_j}), \cdots ,D({y_k}), \cdots ,{y_t}{y_i}, \cdots ,{{\hat y}_i}, \cdots ,{y_{n + 1}}} \right]\\
	&=& 0,
	\end{eqnarray*}	
 which implies ${M_1} =  0.$

Thus, we get
\begin{eqnarray*}
		&&\sum\limits_{i = 1}^{n + 1} {\sum\limits_{j = 1,j \ne i}^{n + 1} {\sum\limits_{k = j + 1,k \ne i}^{n + 1} {{{( - 1)}^i}n{y_i}\left[ {{y_1}, \cdots ,D({y_j}), \cdots ,D({y_k}), \cdots ,{{\hat y}_i}, \cdots ,{y_{n + 1}}} \right]} } } \\
		&= &{M_2} + {M_3}.
		\end{eqnarray*}
	
Therefore we only need to prove the following equation
\begin{eqnarray*}
		&&{M_2} + {M_3} \\
		&=& \sum\limits_{i = 1}^{n + 1} {\sum\limits_{j = 1,j \ne i}^{n + 1} {\sum\limits_{k = 1,k \ne i,j}^{n + 1} {{{( - 1)}^{i - 1}}\left[ {{y_1}, \cdots ,D({y_j}), \cdots ,{y_k}D({y_i}), \cdots ,{{\hat y}_i}, \cdots ,{y_{n + 1}}} \right]} } }.
\end{eqnarray*}	
	
Fristly we have
\begin{eqnarray*}
	&&	\sum\limits_{j = 1,j \ne i}^{n + 1} {\sum\limits_{k = j + 1,k \ne i}^{n + 1} {{{( - 1)}^i}\left[ {{y_1}, \cdots ,{y_i}D({y_j}), \cdots ,D({y_k}), \cdots ,{{\hat y}_i}, \cdots ,{y_{n + 1}}} \right]} } \\
	&&+ \sum\limits_{j = 1,j \ne i}^{n + 1} {\sum\limits_{k = j + 1,k \ne i}^{n + 1} {{{( - 1)}^i}\left[ {{y_1}, \cdots ,D({y_j}), \cdots ,{y_i}D({y_k}), \cdots ,{{\hat y}_i}, \cdots ,{y_{n + 1}}} \right]} } \\
	&=& \sum\limits_{k = 1,k \ne i}^{n + 1} {\sum\limits_{j = 1,j \ne i}^{k - 1} {{{( - 1)}^i}\left[ {{y_1}, \cdots ,{y_i}D({y_j}), \cdots ,D({y_k}), \cdots ,{{\hat y}_i}, \cdots ,{y_{n + 1}}} \right]} } \\
	&&+ \sum\limits_{j = 1,j \ne i}^{n + 1} {\sum\limits_{k = j + 1,k \ne i}^{n + 1} {{{( - 1)}^i}\left[ {{y_1}, \cdots ,D({y_j}), \cdots ,{y_i}D({y_k}), \cdots ,{{\hat y}_i}, \cdots ,{y_{n + 1}}} \right]} } \\
	&=& \sum\limits_{j = 1,j \ne i}^{n + 1} {\sum\limits_{k = 1,k \ne i}^{j - 1} {{{( - 1)}^i}\left[ {{y_1}, \cdots ,{y_i}D({y_k}), \cdots ,D({y_j}), \cdots ,{{\hat y}_i}, \cdots ,{y_{n + 1}}} \right]} } \\
	&&+ \sum\limits_{j = 1,j \ne i}^{n + 1} {\sum\limits_{k = j + 1,k \ne i}^{n + 1} {{{( - 1)}^i}\left[ {{y_1}, \cdots ,D({y_j}), \cdots ,{y_i}D({y_k}), \cdots ,{{\hat y}_i}, \cdots ,{y_{n + 1}}} \right]} }\\
    &= &\sum\limits_{j = 1,j \ne i}^{n + 1} {\sum\limits_{k = 1,k \ne i,j}^{n + 1} {{{( - 1)}^i}\left[ {{y_1}, \cdots ,D({y_j}), \cdots ,{y_i}D({y_k}), \cdots ,{{\hat y}_i}, \cdots ,{y_{n + 1}}} \right]} } .
\end{eqnarray*}	
	
Thus
\begin{eqnarray*}
		&&M_2+M_3 \\
	&	=& \sum\limits_{i = 1}^{n + 1} {\sum\limits_{j = 1,j \ne i}^{n + 1} {\sum\limits_{k = 1,k \ne i,j}^{n + 1} {{{( - 1)}^i}\left[ {{y_1}, \cdots ,D({y_j}), \cdots ,{y_i}D({y_k}), \cdots ,{{\hat y}_i}, \cdots ,{y_{n + 1}}} \right]} } } \\
 &=& \sum\limits_{j = 1}^{n + 1} {\sum\limits_{i = 1,i \ne j}^{n + 1} {\sum\limits_{k = 1,k \ne i,j}^{n + 1} {{{( - 1)}^i}\left[ {{y_1}, \cdots ,D({y_j}), \cdots ,{y_i}D({y_k}), \cdots ,{{\hat y}_i}, \cdots ,{y_{n + 1}}} \right]} } }\\
 	&=& \sum\limits_{j = 1}^{n + 1} {\sum\limits_{i = 1,i \ne j}^{n + 1} {\sum\limits_{k = 1,k \ne j}^{i - 1} {{{( - 1)}^i}\left[ {{y_1}, \cdots ,D({y_j}), \cdots ,{y_i}D({y_k}), \cdots ,{{\hat y}_i}, \cdots ,{y_{n + 1}}} \right]} } } \\
 	&&+ \sum\limits_{j = 1}^{n + 1} {\sum\limits_{i = 1,i \ne j}^{n + 1} {\sum\limits_{k = i + 1,k \ne j}^{n + 1} {{{( - 1)}^i}\left[ {{y_1}, \cdots ,D({y_j}), \cdots ,{{\hat y}_i}, \cdots ,{y_i}D({y_k}), \cdots ,{y_{n + 1}}} \right]} } } .
\end{eqnarray*}

Note that,  for any $1\leq j\leq n+1,$ 	 we have
\begin{eqnarray*}
  	&&\sum\limits_{i = 1,i \ne j}^{n + 1} {\sum\limits_{k = 1,k \ne j}^{i - 1} {{{( - 1)}^i}\left[ {{y_1}, \cdots ,D({y_j}), \cdots ,{y_i}D({y_k}), \cdots ,{{\hat y}_i}, \cdots ,{y_{n + 1}}} \right]} } \\
  	&=& \sum\limits_{i = 1,i \ne j}^{n + 1} {\sum\limits_{k = 1,k \ne j}^{i - 1} {{{( - 1)}^i}\left[ {{y_1}, \cdots ,D({y_j}), \cdots ,{y_{k - 1}},{y_i}D({y_k}),{y_{k + 1}}, \cdots ,{{\hat y}_i} \cdots ,{y_{n + 1}}} \right]} } \\
  	&=& \sum\limits_{i = 1,i \ne j}^{n + 1} {\sum\limits_{k = 1,k \ne j}^{i - 1} {{{( - 1)}^{k - 1}}\left[ {{y_1}, \cdots ,D({y_j}), \cdots ,{{\hat y}_k}, \cdots ,{y_{i - 1}},{y_i}D({y_k}),{y_{i + 1}}, \cdots ,{y_{n + 1}}} \right]} } \\
  	&=& \sum\limits_{i = 1,i \ne j}^{n + 1} {\sum\limits_{k = 1,k \ne j}^{i - 1} {{{( - 1)}^{k - 1}}\left[ {{y_1}, \cdots ,D({y_j}), \cdots ,{{\hat y}_k}, \cdots ,{y_i}D({y_k}), \cdots ,{y_{n + 1}}} \right]} }.
\end{eqnarray*}

Similarly, we have
\begin{eqnarray*}
 	&&\sum\limits_{i = 1,i \ne j}^{n + 1} {\sum\limits_{k = i + 1,k \ne j}^{n + 1} {{{( - 1)}^i}\left[ {{y_1}, \cdots ,D({y_j}), \cdots ,{{\hat y}_i}, \cdots ,{y_i}D({y_k}), \cdots ,{y_{n + 1}}} \right]} } \\
 	&=& \sum\limits_{i = 1,i \ne j}^{n + 1} {\sum\limits_{k = i + 1,k \ne j}^{n + 1} {{{( - 1)}^{k-1}}\left[ {{y_1}, \cdots ,D({y_j}), \cdots ,{y_i}D({y_k}), \cdots ,{{\hat y}_k}, \cdots ,{y_{n + 1}}} \right]} }.
\end{eqnarray*}

Thus
\begin{eqnarray*}
	&&{M_2} + {M_3}\\
	&=& \sum\limits_{j = 1}^{n + 1} {\sum\limits_{i = 1,i \ne j}^{n + 1} {\sum\limits_{k = 1,k \ne j}^{i - 1} {{{( - 1)}^{k - 1}}\left[ {{y_1}, \cdots ,D({y_j}), \cdots ,{{\hat y}_k}, \cdots ,{y_i}D({y_k}), \cdots ,{y_{n + 1}}} \right]} } } \\
	&&+ \sum\limits_{j = 1}^{n + 1} {\sum\limits_{i = 1,i \ne j}^{n + 1} {\sum\limits_{k = i + 1,k \ne j}^{n + 1} {{{( - 1)}^{k-1}}\left[ {{y_1}, \cdots ,D({y_j}), \cdots ,{y_i}D({y_k}), \cdots ,{{\hat y}_k}, \cdots ,{y_{n + 1}}} \right]} } } \\
	&=& \sum\limits_{j = 1}^{n + 1} {\sum\limits_{i = 1,i \ne j}^{n + 1} {\sum\limits_{k = 1,k \ne i,j}^{n + 1} {{{( - 1)}^{k - 1}}\left[ {{y_1}, \cdots ,D({y_j}), \cdots ,{y_i}D({y_k}), \cdots ,{{\hat y}_k}, \cdots ,{y_{n + 1}}} \right]} } }\\
	&=& \sum\limits_{k = 1}^{n + 1} {\sum\limits_{j = 1,j \ne k}^{n + 1} {\sum\limits_{i = 1,i \ne j,k}^{n + 1} {{{( - 1)}^{k - 1}}\left[ {{y_1}, \cdots ,D({y_j}), \cdots ,{y_i}D({y_k}), \cdots ,{{\hat y}_k}, \cdots ,{y_{n + 1}}} \right]} } }\\
&=&\sum\limits_{i = 1}^{n + 1} {\sum\limits_{j = 1,j \ne i}^{n + 1} {\sum\limits_{k = 1,k \ne j,i}^{n + 1} {{{( - 1)}^{i - 1}}\left[ {{y_1}, \cdots ,D({y_j}), \cdots ,{y_k}D({y_i}), \cdots ,{{\hat y}_i}, \cdots ,{y_{n + 1}}} \right]} } } .
\end{eqnarray*}
The proof is completed.
\end{proof}

\begin{theorem}\label{lie3.4}
	Let $(L, \cdot ,\left[ { - , \cdots , - } \right])$ be a  strong transposed Poisson $n$-Lie algebra and  $D$  a derivation of
	$(L, \cdot ,\left[ { - , \cdots , - } \right])$. Define a linear operation on $L$:
	\begin{equation}\label{(3.3)}
		{\mu _{n + 1}}({x_1}, \cdots ,{x_{n + 1}}): = \sum\limits_{i = 1}^{n + 1} {{{( - 1)}^{i - 1}}}D({x_i})[ {x_1}, \cdots ,{{\hat x}_i}, \cdots ,{x_{n + 1}}]
	\end{equation}
for any ${x_i}\in L, 1\leq i\leq n+1.$
	Then $(L,\mu_{n+1} )$ is an $(n+1)$-Lie algebra.
\end{theorem}

\begin{proof}	
For convenience, we denote
	$${\mu _{n + 1}}\left( {{x_1}, \cdots ,{x_{n + 1}}} \right): = \left[ {{x_1}, \cdots ,{x_{n + 1}}} \right].$$

On the one hand, we have
\begin{eqnarray*}
	&&	{\left[ {\left[ {{y_1}, \cdots ,{y_{n + 1}}} \right],{x_1}, \cdots ,{x_n}} \right]}\\
		&\mathop  = \limits^{(\ref{(3.3)})}& \sum\limits_{i = 1}^{n + 1} {{{( - 1)}^{i - 1}}\left[ {D({y_i})\left[ {{y_1}, \cdots ,{{\hat y}_i}, \cdots ,{y_{n + 1}}} \right],{x_1}, \cdots ,{x_n}} \right]} \\
		&\mathop  =\limits^{(\ref{(3.3)})}& \sum\limits_{i = 1}^{n + 1} {{{( - 1)}^{i - 1}}D(D({y_i})\left[ {{y_1}, \cdots ,{{\hat y}_i}, \cdots ,{y_{n + 1}}} \right])\left[ {{x_1}, \cdots ,{x_n}} \right]} \\
		&&{ + \sum\limits_{i = 1}^{n + 1} {\sum\limits_{j = 1}^n {{{( - 1)}^{i + j - 1}}D({x_j})} } \left[ {D({y_i})\left[ {{y_1}, \cdots ,{{\hat y}_i}, \cdots ,{y_{n + 1}}} \right],{x_1}, \cdots ,{{\hat x}_j}, \cdots ,{x_n}} \right]}\\
		&= &\sum\limits_{i = 1}^{n + 1} {{{( - 1)}^{i - 1}}{D^2}({y_i})\left[ {{y_1}, \cdots ,{{\hat y}_i}, \cdots ,{y_{n + 1}}} \right]\left[ {{x_1}, \cdots ,{x_n}} \right]} \\
		&&+ \sum\limits_{i = 1}^{n + 1} {{{( - 1)}^{i - 1}}D({y_i})D(\left[ {{y_1}, \cdots ,{{\hat y}_i}, \cdots ,{y_{n + 1}}} \right])\left[ {{x_1}, \cdots ,{x_n}} \right]}\\
		&&+ \sum\limits_{i = 1}^{n + 1} {\sum\limits_{j = 1}^n {{{( - 1)}^{i + j - 1}}D({x_j})} } \left[ {D({y_i})\left[ {{y_1}, \cdots ,{{\hat y}_i}, \cdots ,{y_{n + 1}}} \right],{x_1}, \cdots ,{{\hat x}_j}, \cdots ,{x_n}} \right]\\
		&\mathop  = \limits^{(\ref{(3.1)})}& \sum\limits_{i = 1}^{n + 1} {{{( - 1)}^{i - 1}}{D^2}({y_i})\left[ {{y_1}, \cdots ,{{\hat y}_i}, \cdots ,{y_{n + 1}}} \right]\left[ {{x_1}, \cdots ,{x_n}} \right]}\\
		&&+ \sum\limits_{i = 1}^{n + 1} {\sum\limits_{k = 1,k \ne i}^{n + 1} {{{( - 1)}^{i - 1}}D({y_i})\left[ {{y_1}, \cdots ,D({y_k}), \cdots ,{{\hat y}_i}, \cdots ,{y_{n + 1}}} \right]\left[ {{x_1}, \cdots ,{x_n}} \right]} } \\
		&&+ \sum\limits_{i = 1}^{n + 1} {\sum\limits_{j = 1}^n {{{( - 1)}^{i + j - 1}}D({x_j})} } \left[ {D({y_i})\left[ {{y_1}, \cdots ,{{\hat y}_i}, \cdots ,{y_{n + 1}}} \right],{x_1}, \cdots ,{{\hat x}_j}, \cdots ,{x_n}} \right]\\	
		&=& \sum\limits_{i = 1}^{n + 1} {{{( - 1)}^{i - 1}}{D^2}({y_i})\left[ {{y_1}, \cdots ,{{\hat y}_i}, \cdots ,{y_{n + 1}}} \right]\left[ {{x_1}, \cdots ,{x_n}} \right]} \\
		&&+ \sum\limits_{k = 1}^{n + 1} {\sum\limits_{i = 1}^{k - 1} {{{( - 1)}^{k + i - 1}}D({y_i})\left[ {D({y_k}),{y_1}, \cdots ,{{\hat y}_i}, \cdots ,{{\hat y}_k}, \cdots ,{y_{n + 1}}} \right]\left[ {{x_1}, \cdots ,{x_n}} \right]} }\\
		&&+ \sum\limits_{k = 1}^{n + 1} {\sum\limits_{i = k + 1}^{n + 1} {{{( - 1)}^{i + k}}D({y_i})\left[ {D({y_k}),{y_1}, \cdots ,{{\hat y}_k}, \cdots ,{{\hat y}_i}, \cdots ,{y_{n + 1}}} \right]\left[ {{x_1}, \cdots ,{x_n}} \right]} } \\
		&&+ \sum\limits_{i = 1}^{n + 1} {\sum\limits_{j = 1}^n {{{( - 1)}^{i + j - 1}}D({x_j})} } \left[ {D({y_i})\left[ {{y_1}, \cdots ,{{\hat y}_i}, \cdots ,{y_{n + 1}}} \right],{x_1}, \cdots ,{{\hat x}_j}, \cdots ,{x_n}} \right].
	\end{eqnarray*}

On the other hand, for any $1\leq k\leq n$, we have
\begin{eqnarray*}
	&&{( - 1)^{k - 1}}\left[ {\left[ {{y_k},{x_1}, \cdots ,{x_n}} \right],{y_1}, \cdots ,{{\hat y}_k}, \cdots ,{y_{n + 1}}} \right]\\
	&\mathop  = \limits^{(\ref{(3.3)})}& {( - 1)^{k - 1}}\left[ {D({y_k})\left[ {{x_1}, \cdots ,{x_n}} \right],{y_1}, \cdots ,{{\hat y}_k}, \cdots ,{y_{n + 1}}} \right]\\
	&&+ \sum\limits_{j = 1}^n {{{( - 1)}^{j + k - 1}}\left[ {D({x_j})\left[ {{y_k},{x_1}, \cdots ,{{\hat x}_j}, \cdots ,{x_n}} \right],{y_1}, \cdots ,{{\hat y}_k}, \cdots ,{y_{n + 1}}} \right]} \\
	&\mathop  = \limits^{(\ref{(3.3)})}& {( - 1)^{k - 1}}D(D({y_k})\left[ {{x_1}, \cdots ,{x_n}} \right])\left[ {{y_1}, \cdots ,{{\hat y}_k}, \cdots ,{y_{n + 1}}} \right]\\
	&&+ \sum\limits_{i = 1}^{k - 1} {{{( - 1)}^{i + k - 1}}D({y_i})\left[ {D({y_k})\left[ {{x_1}, \cdots ,{x_n}} \right],{y_1}, \cdots ,{{\hat y}_i}, \cdots ,{{\hat y}_k}, \cdots ,{y_{n + 1}}} \right]} \\
    &&	+ \sum\limits_{i = k + 1}^{n + 1} {{{( - 1)}^{i + k}}D({y_i})\left[ {D({y_k})\left[ {{x_1}, \cdots ,{x_n}} \right],{y_1}, \cdots ,{{\hat y}_k}, \cdots ,{{\hat y}_i}, \cdots ,{y_{n + 1}}} \right]} \\
	&&+ \sum\limits_{j = 1}^n {{{( - 1)}^{j + k - 1}}D(D({x_j})\left[ {{y_k},{x_1}, \cdots ,{{\hat x}_j}, \cdots ,{x_n}} \right])\left[ {{y_1}, \cdots ,{{\hat y}_k}, \cdots ,{y_{n + 1}}} \right]} \\
	&& +\sum\limits_{j = 1}^n {\sum\limits_{i = k + 1}^{n + 1} {{{(( - 1)}^{ i + j}}D({y_i}) } }\\
	&&\cdot\left[ {D({x_j})\left[ {{y_k},{x_1}, \cdots ,{{\hat x}_j}, \cdots ,{x_n}} \right],{y_1}, \cdots ,{{\hat y}_k}, \cdots ,{{\hat y}_i}, \cdots ,{y_{n + 1}}} \right] )\\
    &&+ \sum\limits_{j = 1}^n {\sum\limits_{i = 1}^{k - 1} {{{(( - 1)}^{i + j - 1}}D({y_i})} }\\
    &&\cdot\left[ {D({x_j})\left[ {{y_k},{x_1}, \cdots ,{{\hat x}_j}, \cdots ,{x_n}} \right],{y_1}, \cdots ,{{\hat y}_i}, \cdots ,{{\hat y}_k}, \cdots ,{y_{n + 1}}} \right])\\
	&=& {( - 1)^{k - 1}}{D^2}({y_k})\left[ {{x_1}, \cdots ,{x_n}} \right]\left[ {{y_1}, \cdots ,{{\hat y}_k}, \cdots ,{y_{n + 1}}} \right]\\
	&&+ \sum\limits_{i = 1}^{k - 1} {{{( - 1)}^{i + k - 1}}D({y_i})\left[ {D({y_k})\left[ {{x_1}, \cdots ,{x_n}} \right],{y_1}, \cdots ,{{\hat y}_i}, \cdots ,{{\hat y}_k}, \cdots ,{y_{n + 1}}} \right]} \\	
	&&+ \sum\limits_{i = k + 1}^{n + 1} {{{( - 1)}^{i + k}}D({y_i})\left[ {D({y_k})\left[ {{x_1}, \cdots ,{x_n}} \right],{y_1}, \cdots ,{{\hat y}_k}, \cdots ,{{\hat y}_i}, \cdots ,{y_{n + 1}}} \right]} \\
	&&+ \sum\limits_{j = 1}^n {{{( - 1)}^{j + k - 1}}{D^2}({x_j})\left[ {{y_k},{x_1}, \cdots ,{{\hat x}_j}, \cdots ,{x_n}} \right]\left[ {{y_1}, \cdots ,{{\hat y}_k}, \cdots ,{y_{n + 1}}} \right]} \\
	&&+ \sum\limits_{j = 1}^n {\sum\limits_{i = k + 1}^{n + 1} {{{(( - 1)}^{ i + j}}D({y_i}) } }\\
	&&\cdot\left[ {D({x_j})\left[ {{y_k},{x_1}, \cdots ,{{\hat x}_j}, \cdots ,{x_n}} \right],{y_1}, \cdots ,{{\hat y}_k}, \cdots ,{{\hat y}_i}, \cdots ,{y_{n + 1}}} \right] )\\
	&&+\sum\limits_{j = 1}^n {\sum\limits_{i = 1}^{k - 1} {{{(( - 1)}^{ i + j - 1}}D({y_i}) } }\\
	&&\cdot\left[ {D({x_j})\left[ {{y_k},{x_1}, \cdots ,{{\hat x}_j}, \cdots ,{x_n}} \right],{y_1}, \cdots ,{{\hat y}_i}, \cdots ,{{\hat y}_k}, \cdots ,{y_{n + 1}}} \right])\\
	&&+ {( - 1)^{k - 1}}D({y_k})D(\left[ {{x_1}, \cdots ,{x_n}} \right])\left[ {{y_1}, \cdots ,{{\hat y}_k}, \cdots ,{y_{n + 1}}} \right]\\
	&&+ \sum\limits_{j = 1}^n {{{( - 1)}^{j + k - 1}}D({x_j})D(\left[ {{y_k},{x_1}, \cdots ,{{\hat x}_j}, \cdots ,{x_n}} \right])\left[ {{y_1}, \cdots ,{{\hat y}_k}, \cdots ,{y_{n + 1}}} \right]} \\
	&\mathop  = \limits^{(\ref{(3.2)})}& {( - 1)^{k - 1}}{D^2}({y_k})\left[ {{x_1}, \cdots ,{x_n}} \right]\left[ {{y_1}, \cdots ,{{\hat y}_k}, \cdots ,{y_{n + 1}}} \right]\\
		\end{eqnarray*}
\begin{eqnarray*}
	&&+ \sum\limits_{i = 1}^{k - 1} {{{( - 1)}^{i + k - 1}}D({y_i})\left[ {D({y_k})\left[ {{x_1}, \cdots ,{x_n}} \right],{y_1}, \cdots ,{{\hat y}_i}, \cdots ,{{\hat y}_k}, \cdots ,{y_{n + 1}}} \right]} \\
	&&+ \sum\limits_{i = k + 1}^{n + 1} {{{( - 1)}^{i + k}}D({y_i})\left[ {D({y_k})\left[ {{x_1}, \cdots ,{x_n}} \right],{y_1}, \cdots ,{{\hat y}_k}, \cdots ,{{\hat y}_i}, \cdots ,{y_{n + 1}}} \right]} \\
	&&+ \sum\limits_{j = 1}^n {{{( - 1)}^{j + k - 1}}{D^2}({x_j})\left[ {{y_k},{x_1}, \cdots ,{{\hat x}_j}, \cdots ,{x_n}} \right]\left[ {{y_1}, \cdots ,{{\hat y}_k}, \cdots ,{y_{n + 1}}} \right]} \\
	&& +\sum\limits_{j = 1}^n {\sum\limits_{i = k + 1}^{n + 1} {{{(( - 1)}^{ i + j}}D({y_i}) } }\\
	&&\cdot\left[ {D({x_j})\left[ {{y_k},{x_1}, \cdots ,{{\hat x}_j}, \cdots ,{x_n}} \right],{y_1}, \cdots ,{{\hat y}_k}, \cdots ,{{\hat y}_i}, \cdots ,{y_{n + 1}}} \right] )\\
	&&+ \sum\limits_{j = 1}^n {\sum\limits_{i = 1}^{k - 1} {{{(( - 1)}^{ i + j - 1}}D({y_i}) } }\\
	&&\cdot\left[ {D({x_j})\left[ {{y_k},{x_1}, \cdots ,{{\hat x}_j}, \cdots ,{x_n}} \right],{y_1}, \cdots ,{{\hat y}_i}, \cdots ,{{\hat y}_k}, \cdots ,{y_{n + 1}}} \right])\\
	&&+ \sum\limits_{j = 1}^n {\sum\limits_{t = j + 1}^n {{{( - 1)}^k}{y_k}\left[ {{x_1}, \cdots ,D({x_j}), \cdots ,D({x_t}), \cdots ,{x_n}} \right]\left[ {{y_1}, \cdots ,{{\hat y}_k}, \cdots ,{y_{n + 1}}} \right]} } \\
	&&+ \sum\limits_{i = 1}^n {\sum\limits_{j = 1,j \ne i}^n {{{( - 1)}^{k + i}}{x_i}\left[ {D({y_k}),{x_1}, \cdots ,D({x_j}), \cdots ,{{\hat x}_i}, \cdots ,{x_n}} \right]\left[ {{y_1}, \cdots ,{{\hat y}_k}, \cdots ,{y_{n + 1}}} \right]} } \\	
	&&+ \sum\limits_{i = 1}^n {\sum\limits_{j = 1,j \ne i}^n {\sum\limits_{t = j + 1,t \ne i}^n {{{( - 1)}^{k + i}}{x_i}\left[ {{y_k},{x_1}, \cdots ,D({x_j}),D({x_t}), \cdots ,{{\hat x}_i}, \cdots ,{x_n}} \right]} } } \\
	&&\cdot \left[ {{y_1}, \cdots ,{{\hat y}_k}, \cdots ,{y_{n + 1}}} \right].
\end{eqnarray*}
	
We denote
	$$
		\sum\limits_{i = 1}^{n + 1} {{{( - 1)}^{i - 1}}\left[ {\left[ {{y_i},{x_1}, \cdots ,{x_n}} \right],{y_1}, \cdots ,{{\hat y}_i}, \cdots ,{y_{n + 1}}} \right]} = \sum\limits_{i = 1}^7 {{A_i}},
	$$	
	where	
\begin{eqnarray*}
	{A_1} &:=& \sum\limits_{i = 1}^{n + 1} {{{( - 1)}^{i - 1}}{D^2}({y_i})\left[ {{x_1}, \cdots ,{x_n}} \right]\left[ {{y_1}, \cdots ,{{\hat y}_i}, \cdots ,{y_{n + 1}}} \right]} \\
	{A_2} &:=& \sum\limits_{k = 1}^{n + 1} {\sum\limits_{j = 1}^n {{{( - 1)}^{k + j - 1}}} {D^2}({x_j})\left[ {{y_k},{x_1}, \cdots ,{{\hat x}_j}, \cdots ,{x_n}} \right]\left[ {{y_1}, \cdots ,{{\hat y}_k}, \cdots ,{y_{n + 1}}} \right]}\\
	{A_3} &:=& \sum\limits_{i = 1}^{n + 1} {\sum\limits_{j = 1}^n {\sum\limits_{k = j + 1}^n {{{( - 1)}^i}{y_i}\left[ {{x_1}, \cdots ,D({x_j}), \cdots ,D({x_k}), \cdots ,{x_n}} \right]\left[ {{y_1}, \cdots ,{{\hat y}_i}, \cdots ,{y_{n + 1}}} \right]} } } \\
		\end{eqnarray*}

\begin{eqnarray*}
		{A_4} &:= &\sum\limits_{k = 1}^{n + 1} {\sum\limits_{i = 1}^n {\sum\limits_{j = 1,j \ne i}^n {\sum\limits_{t = j + 1,t \ne i}^n {({{( - 1)}^{k + i}}{x_i}\left[ {{y_k},{x_1}, \cdots ,D({x_j}),D({x_t}), \cdots ,{{\hat x}_i}, \cdots ,{x_n}} \right]} } } }  \\
	&&	\cdot \left[ {{y_1}, \cdots ,{{\hat y}_k}, \cdots ,{y_{n + 1}}} \right])\\
		{A_5} &:=& \sum\limits_{k = 1}^{n + 1} {\sum\limits_{i = 1}^{k - 1} {{{( - 1)}^{k + i - 1}}D({y_i})\left[ {D({y_k})\left[ {{x_1}, \cdots ,{x_n}} \right],{y_1}, \cdots ,{{\hat y}_i}, \cdots ,{{\hat y}_k}, \cdots ,{y_{n + 1}}} \right]} } \\
		&&+ \sum\limits_{k = 1}^{n + 1} {\sum\limits_{i = k + 1}^{n + 1} {{{( - 1)}^{i + k}}D({y_i})\left[ {D({y_k})\left[ {{x_1}, \cdots ,{x_n}} \right],{y_1}, \cdots ,{{\hat y}_k}, \cdots ,{{\hat y}_i}, \cdots ,{y_{n + 1}}} \right]} }\\
		{A_6}&:=& \sum\limits_{k = 1}^{n + 1} {\sum\limits_{i = 1}^n {\sum\limits_{j = 1,j \ne i}^n {{{(( - 1)}^{k + i}}{x_i}\left[ {D({y_k}),{x_1}, \cdots ,D({x_j}), \cdots ,{{\hat x}_i}, \cdots ,{x_n}} \right]} } } \\
		&&	\cdot\left[ {{y_1}, \cdots ,{{\hat y}_k}, \cdots ,{y_{n + 1}}} \right] )\\
		{A_7}&:=& \sum\limits_{k = 1}^{n + 1} {\sum\limits_{j = 1}^n {\sum\limits_{i = k + 1}^{n + 1} {{{(( - 1)}^{k + i + j}}D({y_i})} } }\\
		&&\cdot\left[ {D({x_j})\left[ {{y_k},{x_1}, \cdots ,{{\hat x}_j}, \cdots ,{x_n}} \right],{y_1}, \cdots ,{{\hat y}_k}, \cdots ,{{\hat y}_i}, \cdots ,{y_{n + 1}}} \right] )\\
		&&+ \sum\limits_{k = 1}^{n + 1} {\sum\limits_{j = 1}^n {\sum\limits_{i = 1}^{k - 1} {{{(( - 1)}^{k + i + j - 1}}D({y_i})} } }\\
		&&\cdot\left[ {D({x_j})\left[ {{y_k},{x_1}, \cdots ,{{\hat x}_j}, \cdots ,{x_n}} \right],{y_1}, \cdots ,{{\hat y}_i}, \cdots ,{{\hat y}_k}, \cdots ,{y_{n + 1}}} \right]) .
\end{eqnarray*}	

By Equation (\ref{(np3)}), for fixed $j$, we have
\[\sum\limits_{k = 1}^{n + 1} {{{( - 1)}^{k + j - 1}}{D^2}({x_j})\left[ {{y_k},{x_1}, \cdots ,{{\hat x}_j}, \cdots ,{x_n}} \right]\left[ {{y_1}, \cdots ,{{\hat y}_k}, \cdots ,{y_{n + 1}}} \right]}  = 0.\]
So  we obtain  $ {{A_2}} =0.$
	
By Equation (\ref{(np1)}), for fixed $j$ and $k$, we have	
	\[\sum\limits_{i = 1}^{n + 1} {{{( - 1)}^i}{y_i}\left[ {{x_1}, \cdots ,D({x_j}), \cdots ,D({x_k}), \cdots ,{x_n}} \right]\left[ {{y_1}, \cdots ,{{\hat y}_i}, \cdots ,{y_{n + 1}}} \right]}  = 0.\]
So  we obtain ${{A_3}}  = 0.$	
	
By Equation (\ref{(np3)}), for fixed $j$ and $t$, we have	
	\[\sum\limits_{k = 1}^{n + 1} {{{( - 1)}^{k + i}}{x_i}\left[ {{y_k},{x_1}, \cdots ,D({x_j}),D({x_t}), \cdots ,{{\hat x}_i}, \cdots ,{x_n}} \right]\left[ {{y_1}, \cdots ,{{\hat y}_k}, \cdots ,{y_{n + 1}}} \right]}  = 0.\]
So  we obtain  $  {{A_4}} =0.$
	
By Equation (\ref{(strong1)}), for fixed $i$ and $k$,  we have
\begin{eqnarray*}
	&&{( - 1)^{k + i - 1}}D({y_i})\left[ {D({y_k})\left[ {{x_1}, \cdots ,{x_n}} \right],{y_1}, \cdots ,{{\hat y}_i}, \cdots ,{{\hat y}_k}, \cdots ,{y_{n + 1}}} \right]\\
	&&+ {( - 1)^{i + k}}D({y_k})\left[ {D({y_i})\left[ {{x_1}, \cdots ,{x_n}} \right],{y_1}, \cdots ,{{\hat y}_i}, \cdots ,{{\hat y}_k}, \cdots ,{y_{n + 1}}} \right]\\
	&=& {( - 1)^{k + i - 1}}D({y_i})\left[ {D({y_k}),{y_1}, \cdots ,{{\hat y}_i}, \cdots ,{{\hat y}_k}, \cdots ,{y_{n + 1}}} \right]\left[ {{x_1}, \cdots ,{x_n}} \right]\\
	&&+ {( - 1)^{i + k}}D({y_k})\left[ {D({y_i}),{y_1}, \cdots ,{{\hat y}_i}, \cdots ,{{\hat y}_k}, \cdots ,{y_{n + 1}}} \right]\left[ {{x_1}, \cdots ,{x_n}} \right].
\end{eqnarray*}

Thus, we obtain
\begin{eqnarray*}
	{{A_5}} & = &\sum\limits_{k = 1}^{n + 1} {\sum\limits_{i = 1}^{k - 1} {{{( - 1)}^{k + i - 1}}D({y_i})\left[ {D({y_k}),{y_1}, \cdots ,{{\hat y}_i}, \cdots ,{{\hat y}_k}, \cdots ,{y_{n + 1}}} \right]\left[ {{x_1}, \cdots ,{x_n}} \right]} } \\
	&&	+ \sum\limits_{k = 1}^{n + 1} {\sum\limits_{i = k + 1}^{n + 1} {{{( - 1)}^{i + k}}D({y_i})\left[ {D({y_k}),{y_1}, \cdots ,{{\hat y}_k}, \cdots ,{{\hat y}_i}, \cdots ,{y_{n + 1}}} \right]\left[ {{x_1}, \cdots ,{x_n}} \right]} }.
\end{eqnarray*}	

By Equation (\ref{(np1)}), for fixed $j$ and $k$, we have
\begin{eqnarray*}
	&&\sum\limits_{i = 1}^n {{{( - 1)}^{k + i}}{x_i}\left[ {D({y_k}),{x_1}, \cdots ,D({x_j}), \cdots ,{{\hat x}_i}, \cdots ,{x_n}} \right]} \\
	&=& {( - 1)^{k - 1}}D({y_k})\left[ {{x_1}, \cdots ,D({x_j}), \cdots ,{x_n}} \right] + {( - 1)^{k + j - 1}}D({x_j})\left[ {D({y_k}),{x_1}, \cdots ,{x_n}} \right]\\
	&=& {( - 1)^{k + j}}D({y_k})\left[ {D({x_j}),{x_1}, \cdots ,{{\hat x}_j}, \cdots ,{x_n}} \right] + {( - 1)^{k + j - 1}}D({x_j})\left[ {D({y_k}),{x_1}, \cdots ,{x_n}} \right].
\end{eqnarray*}	
Thus, we get
	\begin{eqnarray*}
		 {{A_6}} &	= &\sum\limits_{k = 1}^{n + 1} {\sum\limits_{j = 1}^n {{{( - 1)}^{k + j}}D({y_k})\left[ {D({x_j}),{x_1}, \cdots ,{{\hat x}_j}, \cdots ,{x_n}} \right]\left[ {{y_1}, \cdots ,{{\hat y}_k}, \cdots ,{y_{n + 1}}} \right]} } \\
		&&+ \sum\limits_{k = 1}^{n + 1} {\sum\limits_{j = 1}^n {{{( - 1)}^{k + j - 1}}D({x_j})\left[ {D({y_k}),{x_1}, \cdots ,{x_n}} \right]\left[ {{y_1}, \cdots ,{{\hat y}_k}, \cdots ,{y_{n + 1}}} \right]} } .
\end{eqnarray*}

By Equation (\ref{(np2)}), for fixed $j$ and $i$, we have
	\begin{eqnarray*}
	&&	\sum\limits_{k = i + 1}^{n + 1} {{{( - 1)}^{k + i + j - 1}}D({y_i})\left[ {D({x_j})\left[ {{y_k},{x_1}, \cdots ,{{\hat x}_j}, \cdots ,{x_n}} \right],{y_1}, \cdots ,{{\hat y}_i}, \cdots ,{{\hat y}_k}, \cdots ,{y_{n + 1}}} \right]} \\
	&&	+ \sum\limits_{k = 1}^{i - 1} {{{( - 1)}^{k + i + j}}D({y_i})\left[ {D({x_j})\left[ {{y_k},{x_1}, \cdots ,{{\hat x}_j}, \cdots ,{x_n}} \right],{y_1}, \cdots ,{{\hat y}_k}, \cdots ,{{\hat y}_i}, \cdots ,{y_{n + 1}}} \right]} \\
		&=& {( - 1)^{j + i - 1}}D({y_i})\left[ {D({x_j})\left[ {{y_1}, \cdots ,{{\hat y}_i}, \cdots ,{y_{n + 1}}} \right],{x_1}, \cdots ,{{\hat x}_j}, \cdots ,{x_n}} \right].
	\end{eqnarray*}
So we obtain
	\[ {{{A_7}}} = \sum\limits_{j = 1}^n {\sum\limits_{i = 1}^{n + 1} {{{( - 1)}^{j + i - 1}}D({y_i})\left[ {D({x_j})\left[ {{y_1}, \cdots ,{{\hat y}_i}, \cdots ,{y_{n + 1}}} \right],{x_1}, \cdots ,{{\hat x}_j}, \cdots ,{x_n}} \right]} } .\]
	
By Equation (\ref{(strong1)}),  we have
\begin{eqnarray*}
		&&{( - 1)^{i + j}}D({y_i})\left[ {D({x_j}),{x_1}, \cdots ,{{\hat x}_j}, \cdots ,{x_n}} \right]\left[ {{y_1}, \cdots ,{{\hat y}_i}, \cdots ,{y_{n + 1}}} \right]\\
		&&+ {( - 1)^{i + j - 1}}D({x_j})\left[ {D({y_i}),{x_1}, \cdots ,{x_n}} \right]\left[ {{y_1}, \cdots ,{{\hat y}_i}, \cdots ,{y_{n + 1}}} \right]\\
		&&+ {( - 1)^{j + i - 1}}D({y_i})\left[ {D({x_j})\left[ {{y_1}, \cdots ,{{\hat y}_i}, \cdots ,{y_{n + 1}}} \right],{x_1}, \cdots ,{{\hat x}_j}, \cdots ,{x_n}} \right]\\
		&=& {( - 1)^{j + i - 1}}D({x_j})\left[ {D({y_i})\left[ {{y_1}, \cdots ,{{\hat y}_i}, \cdots ,{y_{n + 1}}} \right],{x_1}, \cdots ,{{\hat x}_j}, \cdots ,{x_n}} \right].
	\end{eqnarray*}
So we get
	$${{{A_6}}}  +  {{{A_7}}}  = \sum\limits_{i = 1}^{n + 1} {\sum\limits_{j = 1}^n {{{( - 1)}^{j + i - 1}}D({x_j})\left[ {D({y_i})\left[ {{y_1}, \cdots ,{{\hat y}_i}, \cdots ,{y_{n + 1}}} \right],{x_1}, \cdots ,{{\hat x}_j}, \cdots ,{x_n}} \right].} } $$

Thus, we have
	$$
	\sum\limits_{i = 1}^7 {{A_i}} 	=  {{A_1}}  +{{A_5}} + {{A_6}}  + {{A_7}} =\left[ {\left[ {{y_1}, \cdots ,{y_{n + 1}}} \right],{x_1}, \cdots ,{x_n}} \right].
	$$
Therefore, $(L,\mu_{n+1} )$ is an $(n+1)$-Lie algebra. 	
\end{proof}

Now  we can prove Conjecture \ref{con} for strong transposed Poisson $n$-Lie algebras.
	
\begin{theorem}\label{main}
With the notations in Theorem \ref{lie3.4}, $(L, \cdot ,\mu_{n+1} )$ is a strong transposed Poisson
	$(n + 1)$-Lie algebra.	
\end{theorem}
\begin{proof}
For convenience, we denote ${\mu _{n + 1}}\left( {{x_1}, \cdots ,{x_{n + 1}}} \right): = \left[ {{x_1}, \cdots ,{x_{n + 1}}} \right].$
According to Theorem \ref{lie3.4}, we only need to prove Equation (\ref{(np)}) and Equation (\ref{(strong)}).

{\emph{Proof of Equation (\ref{(np)})}}. By Equation (\ref{(3.3)}), we have
\begin{eqnarray*}
	&&\sum\limits_{i = 1}^{n + 1} {\left[ {{x_1}, \cdots ,h{x_i}, \cdots ,{x_{n + 1}}} \right]} \\
	&=& D(h{x_1})\left[ {{x_2}, \cdots ,{x_{n + 1}}} \right] + \sum\limits_{j = 2}^{n + 1} {{{( - 1)}^{j - 1}}D({x_j})\left[ {h{x_1},{x_2}, \cdots ,{{\hat x}_j}, \cdots ,{x_{n + 1}}} \right]} \\
	&&- D(h{x_2})\left[ {{x_1},{x_3}, \cdots ,{x_{n + 1}}} \right] + \sum\limits_{j = 1,j \ne 2}^{n + 1} {{{( - 1)}^{j - 1}}D({x_j})\left[ {{x_1},h{x_2},{x_3}, \cdots ,{{\hat x}_j}, \cdots ,{x_{n + 1}}} \right]} \\
	&&+  \cdots  + {( - 1)^n}D(h{x_n})\left[ {{x_1}, \cdots ,{x_n}} \right] + \sum\limits_{j = 1}^n {{{( - 1)}^{j - 1}}D({x_j})\left[ {{x_1}, \cdots ,{{\hat x}_j}, \cdots ,{x_n},h{x_{n + 1}}} \right]} \\
		\end{eqnarray*}
\begin{eqnarray*}
	&=& \sum\limits_{i = 1}^{n + 1} {{{( - 1)}^{i - 1}}D(h{x_i})\left[ {{x_1}, \cdots ,{{\hat x}_i}, \cdots ,{x_n}} \right]} \\
	&&+\sum\limits_{i = 1}^{n + 1} {\sum\limits_{j = 1,j \ne i}^{n + 1} {{{( - 1)}^{j - 1}}D({x_j})\left[ {{x_1},, \cdots ,h{x_i}, \cdots ,{{\hat x}_j}, \cdots ,{x_{n + 1}}} \right]} } \\
	&=& \sum\limits_{i = 1}^{n + 1} {{{( - 1)}^{i - 1}}D(h{x_i})} \left[ {{x_1}, \cdots ,{{\hat x}_i}, \cdots ,{x_{n + 1}}} \right]\\
	&&+ \sum\limits_{j = 1}^{n + 1} {\sum\limits_{i = 1,i \ne j}^{n + 1} {{{( - 1)}^{j - 1}}D({x_j})\left[ {{x_1}, \cdots ,h{x_i}, \cdots ,{{\hat x}_j}, \cdots ,{x_{n + 1}}} \right]} } \\
	&=& \sum\limits_{i = 1}^{n + 1} {{{( - 1)}^{i - 1}}hD({x_i})} \left[ {{x_1}, \cdots ,{{\hat x}_i}, \cdots ,{x_{n + 1}}} \right] \\
	&&+ \sum\limits_{i = 1}^{n + 1} {{{( - 1)}^{i - 1}}{x_i}D(h)} \left[ {{x_1}, \cdots ,{{\hat x}_i}, \cdots ,{x_{n + 1}}} \right]\\
	&&	+ \sum\limits_{j = 1}^{n + 1} {\sum\limits_{i = 1,i \ne j}^{n + 1} {{{( - 1)}^{j - 1}}D({x_j})\left[ {{x_1}, \cdots ,h{x_i}, \cdots ,{{\hat x}_j}, \cdots ,{x_{n + 1}}} \right]} } \\
		&	\mathop  = \limits^{(\ref{(np1)})}& \sum\limits_{i = 1}^{n + 1} {{{( - 1)}^{i - 1}}hD({x_i})} \left[ {{x_1}, \cdots ,{{\hat x}_i}, \cdots ,{x_{n + 1}}} \right] \\
	&&+\sum\limits_{j = 1}^{n + 1} {\sum\limits_{i = 1,i \ne j}^{n + 1} {{{( - 1)}^{j - 1}}D({x_j})\left[ {{x_1}, \cdots ,h{x_i}, \cdots ,{{\hat x}_j}, \cdots ,{x_{n + 1}}} \right]} } \\
	&\mathop  = \limits^{(\ref{(3.3)})}& h\left[ {{x_1}, \cdots ,{x_{n + 1}}} \right] + \sum\limits_{j = 1}^{n + 1} {\sum\limits_{i = 1,i \ne j}^{n + 1} {{{( - 1)}^{j - 1}}D({x_j})\left[ {{x_1}, \cdots ,h{x_i}, \cdots ,{{\hat x}_j}, \cdots ,{x_{n + 1}}} \right]} } \\
		&\mathop  = \limits^{(\ref{(np)})}& h\left[ {{x_1}, \cdots ,{x_{n + 1}}} \right] + nh\sum\limits_{j = 1}^{n + 1} {{{( - 1)}^{j - 1}}D({x_j})\left[ {{x_1}, \cdots ,{{\hat x}_j}, \cdots ,{x_{n + 1}}} \right]} \\
	&\mathop  = \limits^{(\ref{(3.3)})}& h\left[ {{x_1}, \cdots ,{x_{n + 1}}} \right] + nh\left[ {{x_1}, \cdots ,{x_{n + 1}}} \right]\\
	&=& (n + 1)h\left[ {{x_1}, \cdots ,{x_{n + 1}}} \right].
\end{eqnarray*}

{\emph{Proof of Equation (\ref{(strong)})}}.  By Equation (\ref{(3.3)}), we have
\begin{eqnarray*}
		&&{y_1}\left[ {h{y_2},{x_1}, \cdots ,{x_n}} \right] - {y_2}\left[ {h{y_1},{x_1}, \cdots ,{x_n}} \right] + \sum\limits_{i = 1}^n {{{( - 1)}^{i - 1}}h{x_i}\left[ {{y_1},{y_2},{x_1}, \cdots ,{{\hat x}_i}, \cdots ,{x_n}} \right]} \\
	&= &{y_1}{y_2}D(h)\left[ {{x_1}, \cdots ,{x_n}} \right] + {y_1}hD({y_2})\left[ {{x_1}, \cdots ,{x_n}} \right] - {y_1}D({x_1})\left[ {h{y_2},{x_2}, \cdots ,{x_n}} \right]\\
		\end{eqnarray*}
\begin{eqnarray*}
	&&+ {y_1}D({x_2})\left[ {h{y_2},{x_1},{x_3}, \cdots ,{x_n}} \right] +  \cdots  + {( - 1)^n}{y_1}D({x_n})\left[ {h{y_2},{x_1}, \cdots ,{x_{n - 1}}} \right]\\
		&&- {y_2}{y_1}D(h)\left[ {{x_1}, \cdots ,{x_n}} \right] - {y_2}hD({y_1})\left[ {{x_1}, \cdots ,{x_n}} \right] + {y_2}D({x_1})\left[ {h{y_1},{x_2}, \cdots ,{x_n}} \right]\\
		&&- {y_2}D({x_2})\left[ {h{y_1},{x_1},{x_3}, \cdots ,{x_n}} \right] +  \cdots  + {( - 1)^{n - 1}}{y_2}D({x_n})\left[ {h{y_1},{x_1}, \cdots ,{x_{n - 1}}} \right]\\
&&	+ h{x_1}D({y_1})\left[ {{y_2},{x_2}, \cdots ,{x_n}} \right] - h{x_1}D({y_2})\left[ {{y_1},{x_2}, \cdots ,{x_n}} \right]\\
	&&+ h{x_1}D({x_2})\left[ {{y_1},{y_2},{x_3}, \cdots ,{x_n}} \right]+  \cdots  + {( - 1)^{n + 1}}h{x_1}D({x_n})\left[ {{y_1},{y_2},{x_2}, \cdots ,{x_{n - 1}}} \right]\\
	&&- h{x_2}D({y_1})\left[ {{y_2},{x_1},{x_3}, \cdots ,{x_n}} \right] + h{x_2}D({y_2})\left[ {{y_1},{x_1},{x_3}, \cdots ,{x_n}} \right] \\
	&&- h{x_2}D({x_1})\left[ {{y_1},{y_2},{x_3}, \cdots ,{x_n}} \right]+  \cdots  + {( - 1)^{n + 2}}h{x_2}D({x_n})\left[ {{y_1},{y_2},{x_1},{x_3}, \cdots ,{x_{n - 1}}} \right]\\
	&&+	\cdots + {( - 1)^{n - 1}}h{x_n}D({y_1})\left[ {{y_2},{x_1}, \cdots ,{x_{n - 1}}} \right] + {( - 1)^n}h{x_n}D({y_2})\left[ {{y_1},{x_1}, \cdots ,{x_{n - 1}}} \right]\\
	&&+ {( - 1)^{n + 1}}h{x_n}D({x_1})\left[ {{y_1},{y_2},{x_2}, \cdots ,{x_{n - 1}}} \right] \\
	&&+  \cdots  + {( - 1)^{2n - 1}}h{x_n}D({x_{n - 1}})\left[ {{y_1},{y_2},{x_1}, \cdots ,{x_{n - 2}}} \right]\\
	&= & - {y_2}hD({y_1})\left[ {{x_1}, \cdots ,{x_n}} \right] + h{x_1}D({y_1})\left[ {{y_2},{x_2}, \cdots ,{x_n}} \right]\\
	&&+ \sum\limits_{i = 2}^n {{{( - 1)}^{i - 1}}h{x_i}D({y_1})\left[ {{y_2},{x_1}, \cdots ,{{\hat x}_i}, \cdots ,{x_n}} \right]} \\
	&&+ {y_1}hD({y_2})\left[ {{x_1}, \cdots ,{x_n}} \right] - h{x_1}D({y_2})\left[ {{y_1},{x_2}, \cdots ,{x_n}} \right] \\
	&&+ \sum\limits_{i = 2}^n {{{( - 1)}^i}h{x_i}D({y_2})\left[ {{y_1},{x_1}, \cdots ,{{\hat x}_i}, \cdots ,{x_n}} \right]}\\
	&&- {y_1}D({x_1})\left[ {h{y_2},{x_2}, \cdots ,{x_n}} \right] + {y_2}D({x_1})\left[ {h{y_1},{x_2}, \cdots ,{x_n}} \right]\\
	&&+ \sum\limits_{i = 2}^n {{{( - 1)}^{i - 1}}h{x_i}D({x_1})\left[ {{y_1},{y_2},{x_2}, \cdots ,{{\hat x}_i}, \cdots ,{x_n}} \right]} \\
	&&+ {y_1}D({x_2})\left[ {h{y_2},{x_1},{x_3}, \cdots ,{x_n}} \right] - {y_2}D({x_2})\left[ {h{y_1},{x_1},{x_3}, \cdots ,{x_n}} \right]\\
	&&+ h{x_1}D({x_2})\left[ {{y_1},{y_2},{x_3}, \cdots ,{x_n}} \right] + \sum\limits_{i = 3}^n {{{( - 1)}^i}h{x_i}D({x_2})\left[ {{y_1},{y_2},{x_1},{x_3}, \cdots ,{{\hat x}_i}, \cdots ,{x_n}} \right]} \\	
	&&\cdots+ {( - 1)^n}{y_1}D({x_n})\left[ {h{y_2},{x_1}, \cdots ,{x_{n - 1}}} \right] + {( - 1)^{n - 1}}{y_2}D({x_n})\left[ {h{y_1},{x_1}, \cdots ,{x_{n - 1}}} \right]\\
	&&+ \sum\limits_{j = 1}^{n - 1} {{{( - 1)}^{n + j - 1}}h{x_j}D({x_n})\left[ {{y_1},{y_2},{x_1}, \cdots ,{{\hat x}_j}, \cdots ,{x_{n - 1}}} \right]} \\
	&=& {A_1} + {A_2} + \sum\limits_{i = 1}^n {{B_i}},
\end{eqnarray*}
where
\begin{eqnarray*}
	{A_1} &:=& - {y_2}hD({y_1})\left[ {{x_1}, \cdots ,{x_n}} \right] + \sum\limits_{i = 1}^n {{{( - 1)}^{i -1}}h{x_i}D({y_1})\left[ {{y_2},{x_1}, \cdots ,{{\hat x}_i}, \cdots ,{x_n}} \right]},\\
	{A_2} &:=& {y_1}hD({y_2})\left[ {{x_1}, \cdots ,{x_n}} \right] + \sum\limits_{i = 1}^n {{{( -1)}^i}h{x_i}D({y_2})\left[ {{y_1},{x_1}, \cdots ,{{\hat x}_i}, \cdots ,{x_n}} \right]},\\
		\end{eqnarray*}
and for any $1\leq i\leq n,$
\begin{eqnarray*}
	{B_i} &:= &{( - 1)^i}{y_1}D({x_i})\left[ {h{y_2},{x_1}, \cdots ,{{\hat x}_i}, \cdots ,{x_n}} \right] + {( - 1)^{i - 1}}{y_2}D({x_i})\left[ {h{y_1},{x_1}, \cdots ,{{\hat x}_i}, \cdots ,{x_n}} \right]\\
	&&+ \sum\limits_{j = 1}^{i - 1} {{{( - 1)}^{i + j - 1}}h{x_j}D({x_i})\left[ {{y_1},{y_2},{x_1}, \cdots ,{{\hat x}_j}, \cdots ,{{\hat x}_i}, \cdots ,{x_n}} \right]} \\
	&&+ \sum\limits_{j = i + 1}^n {{{( - 1)}^{i + j}}h{x_j}D({x_i})\left[ {{y_1},{y_2},{x_1}, \cdots ,{{\hat x}_i}, \cdots ,{{\hat x}_j}, \cdots ,{x_n}} \right]} .
\end{eqnarray*}

By Equation (\ref{(np1)}), we have
\begin{eqnarray*}
	{A_1}& = & hD({y_1})\left( { - {y_2}\left[ {{x_1}, \cdots ,{x_n}} \right] + \sum\limits_{i = 1}^n {{{( - 1)}^{i - 1}}{x_i}\left[ {{y_2},{x_1}, \cdots ,{{\hat x}_i}, \cdots ,{x_n}} \right]} } \right)= 0.
\end{eqnarray*}
Similarly,  we have ${A_2} =0.$

 By Equation (\ref{(strong)}), for any $1\leq i\leq n$, we have
\begin{eqnarray*}
		{B_i} &= &{( - 1)^i}D({x_i})({y_1}\left[ {h{y_2},{x_1}, \cdots ,{{\hat x}_i}, \cdots ,{x_n}} \right] - {y_2}\left[ {h{y_1},{x_1}, \cdots ,{{\hat x}_i}, \cdots ,{x_n}} \right]\\
		&&+ \sum\limits_{j = 1}^{i - 1} {{{( - 1)}^{j - 1}}h{x_j}\left[ {{y_1},{y_2},{x_1}, \cdots ,{{\hat x}_j}, \cdots ,{{\hat x}_i}, \cdots ,{x_n}} \right]} \\
		&&+ \sum\limits_{j = i + 1}^n {{{( - 1)}^j}h{x_j}\left[ {{y_1},{y_2},{x_1}, \cdots ,{{\hat x}_i}, \cdots ,{{\hat x}_j}, \cdots ,{x_n}} \right]} )\\
	&	=& 0.
\end{eqnarray*}

Thus, we get
\[{y_1}\left[ {h{y_2},{x_1}, \cdots ,{x_n}} \right] - {y_2}\left[ {h{y_1},{x_1}, \cdots ,{x_n}} \right] + \sum\limits_{i = 1}^n {{{( - 1)}^{i - 1}}h{x_i}\left[ {{y_1},{y_2},{x_1}, \cdots ,{{\hat x}_i}, \cdots ,{x_n}} \right]}  = 0.\]
The proof is completed.		
\end{proof}

	\section*{Acknowledgments}
	
	Ming Ding was supported by Guangdong Basic and Applied Basic Research Foundation (2023A1515011739) and  the Basic Research Joint Funding Project of University and Guangzhou City under Grant 202201020103.

\end{document}